\documentclass[12pt]{article}
\usepackage{amssymb, amscd, amsthm, amsmath,latexsym, amstext,mathrsfs,verbatim}
\usepackage{pdfpages}
\usepackage{url}
\usepackage[all]{xy}
\usepackage{enumerate}

\def\alg{{A}}
\def\np{\bigskip}
\def\op{{\rm op}}
\def\ddA{{\rm A}}

\def\Br{{\rm Br}}
\def\SBr{{\rm SBr}}
\def\BrD{{\rm{ BrD}}}
\def\BrM{{\rm BrM}}

\def\BrMD{{\rm BrMD}}
\def\ddB{{\rm B}}
\def\ddC{{\rm C}}
\def\ddD{{\rm D}}

\def\eps{{\epsilon}}

\def\SBrM{{\rm SBrM}}

\def\nl{\smallskip\noindent}

\def\lijntje{\vrule height2.4pt depth-2pt width0.5in}

\def\vlijntje{\vrule height0.45in depth0.4pt width0.4pt}
\def\vlijn{\buildrel {\hbox to 0pt{\hss$\textstyle\circ$\hss}}\over\vlijntje}

\def\dlijntje{{\vrule height2pt depth-1.6pt
width0.5in}\llap{\vrule height4pt depth-3.6pt width0.5in}}
\def\vtriple#1\over#2\over#3{\mathrel{\mathop{\kern0pt #2}\limits_{\hbox
to 0pt{\hss$#1$\hss}}^{\hbox to 0pt{\hss$#3$\hss}}}}
\def\rvtriple#1\over#2\over#3{\mathrel{\mathop{\kern0pt #2}\limits_{\hbox
to 0pt{\hss$#3$\hss}}^{\hbox to 0pt{\hss$#1$\hss}}}}

\def\Dn{\vtriple{\scriptstyle n+1}\over\circ\over{}\kern-1pt\lijntje\kern-1pt
\vtriple{\scriptstyle{n}}\over\circ\over{}
\cdots\cdots\vtriple{\scriptstyle 4}\over\circ\over{}\kern-1pt\lijntje\kern-1pt
\vtriple{\scriptstyle 3}\over\circ\over{\buildrel
{\scriptstyle 2}\over\vlijn}\kern-1pt\lijntje\kern-1pt
\vtriple{\scriptstyle 1}\over\circ\over{}\kern-1pt}
\def\An{\vtriple{\scriptstyle 1}\over\circ\over{}\kern-1pt\lijntje\kern-1pt
\vtriple{\scriptstyle{2}}\over\circ\over{}\kern-1pt\lijntje\kern-1pt
\vtriple{\scriptstyle 3}\over\circ\over{}
\cdots\cdots
\vtriple{\scriptstyle 2}\over\circ\over{}\kern-1pt\lijntje\kern-1pt
\vtriple{\scriptstyle 1}\over\circ\over{}\kern-1pt}
\def\Cn{\vtriple{\scriptstyle n-1}\over\circ\over{}
\kern-1pt\lijntje\kern-1pt\vtriple{\scriptstyle{n-2}}\over\circ\over{}
\cdots\cdots
\vtriple{\scriptstyle 2}\over\circ\over{}
\kern-1pt\lijntje\kern-1pt\vtriple{\scriptstyle 1}\over\circ\over{}
\kern-4pt{\dlijntje \kern -25pt>}\kern8pt
\vtriple{\scriptstyle 0}\over\circ\over{}\kern-1pt}

\newcommand{\cA}{\mathcal{A}}

\newcommand{\N}{\mathbb N}

\newcommand{\R}{\mathbb R}
\newcommand{\Z}{\mathbb Z}

\newcommand{\fp}{\mathfrak{p}}

\numberwithin{equation}{section}

\newtheorem{lemma}{Lemma}[section]
\newtheorem{cor}[lemma]{Corollary}
\newtheorem{prop}[lemma]{Proposition}
\newtheorem{thm}[lemma]{Theorem}

\theoremstyle{definition}

\newtheorem{defn}[lemma]{Definition}
\newtheorem{Notation}[lemma]{Notation}

\theoremstyle{remark}
\newtheorem{rem}[lemma]{Remark}

\def\b{\beta}

\begin{document}
\title{Brauer algebras of type B}
\author{Arjeh M.~Cohen, Shoumin Liu}
\maketitle

\begin{abstract}
For each $n\ge1$, we define an algebra having many properties that one
might expect to hold for a Brauer algebra of type $\ddB_{n}$. It is defined
by means of a presentation by generators and relations.  We show that this
algebra is a subalgebra of the Brauer algebra of type $\ddD_{n+1}$ and point
out a cellular structure in it.  This work is a natural sequel to the
introduction of Brauer algebras of type $\ddC_n$, which are subalgebras of
classical Brauer algebras of type $\ddA_{2n-1}$ and differ from the current
ones for $n>2$.
\end{abstract}

\section{Introduction}\label{introduction}
In \cite{CFW2008}, the Brauer algebra $\Br(Q)$ of any simply-laced Coxeter
type was defined in such a way that for $Q=\ddA_{n-1}$, the classical Brauer
algebra of diagrams on $2n$ nodes emerges. For these algebras, a deformation
to a Birman--Murakami--Wenzl (BMW) algebra was defined and in
\cite{CGW2005}, and, for the spherical types among these, the algebra
structure was fully determined in \cite{CGW2008,CW2011}. Again, for
$Q=\ddA_{n-1}$, the classical BMW algebras re-appear. 

The starting point for an extension of these algebras to non-simply laced
diagrams, begun in \cite{CLY2010}, is based on the following obeservation.
It is well known that the Coxeter group of type $\ddB_n$ arises from the
Coxeter group of type $\ddD_{n+1}$ as the subgroup of all elements fixed by
the nontrivial Coxeter diagram automorphism.  Crisp \cite{Crisp1996} showed
that the Artin group of type $\ddB_{n}$ arises in a similar fashion from the
Artin group of type $\ddD_{n+1}$. In this paper, we study the subalgebra
$\SBr(\ddD_{n+1})$ of the Brauer algebra $\Br(\ddD_{n+1})$ spanned by the
monomials fixed under the automorphism induced by the nontrivial Coxeter
diagram automorphism. We also give a presentation of this subalgebra by
generators and relations, which we regard as the definition of a Brauer
algebra of type $\ddB_n$. This paper continues the introduction in
\cite{CLY2010} of a Brauer algebra of type $\ddC_{n}$ of $\Br(\ddA_{2n-1})$
spanned by monomials fixed under the canonical Coxeter diagram automorphism.

Each defining relation (given in Definition \ref{0.1} below) concerns at
most two indices, say $i$ and $j$, and is (up to the parameters in the
idempotent relation) determined by the diagram induced by $\ddB_n$ on
$\{i,j\}$.  The nodes of the Dynkin type $\ddB_n$ are labeled as follows.
$$ \ddB_n\quad = \quad \Cn .$$ The generators of the Brauer algebra
$\Br(\ddB_n)$ are denoted $r_0$,$\ldots$, $r_{n-1}$, $e_0$,$\ldots$,
$e_{n-1}$.  In order to distinguish these from the canonical generators of
the Brauer algebra of type $\ddD_{n+1}$, the latter are denoted
$R_1,\ldots,R_{n+1}$, $E_1,\ldots,E_{n+1}$ instead of the usual lower case
letters (see Definition \ref{1.1}). The diagram for $\ddD_{n+1}$ is depicted
below.
$$\ddD_{n+1}\quad = \quad\Dn.$$

\begin{defn} Let $\sigma$ be the natural isomorphism on $\BrM(\ddD_{n+1})$ (Definition \ref{1.1}), which is
induced by the action of the permutation $(1,2)$ on the indices of the
generators $\{R_i,\,E_i\}_{i=1}^{n+1}$ of $\BrM(\ddD_{n+1})$ and keeping the
parameter $\delta$ invariant. The fixed submonoid of $\BrM(\ddD_{n+1})$
under $\sigma$ is called the \emph{symmetric submonoid} of
$\BrM(\ddD_{n+1})$, and denoted $\SBrM(\ddD_{n+1})$. The linear span of
$\SBrM(\ddD_{n+1})$ is called the \emph{symmetric subalgebra} of
$\Br(\ddD_{n+1})$, and denoted $\SBr(\ddD_{n+1})$.
\end{defn}

\begin{thm}\label{thm:main}
There exists a $\Z[\delta^{\pm1}]$-algebra isomorphism
$$\phi : \Br(\ddB_n)\longrightarrow \SBr(\ddD_{n+1})$$
determined by $\phi(r_0)=R_1R_2$, $\phi(r_i)=R_{i+2}$,
$\phi(e_0)=E_1E_2$,  and $\phi(e_i)=E_{i+2}$, for $0< i\leq n-1$.
Furthermore both algebras  are free of  rank $$f(n):=2^{n+1}\cdot n!!- 2^{n}\cdot n!+(n+1)!!-(n+1)!.$$
\end{thm}

\np In Theorem \ref{th:cellular} of this paper, we also show that these
algebras are cellular in the sense of Graham and Lehrer \cite{GL1996}.  The
subalgebra of $\Br(\ddB_n)$ generated by $r_0,\ldots,r_{n-1}$ is easily seen
to be isomorphic to the group algebra of the Weyl group $W(\ddB_n)$ of type
$\ddB_n$.

Here we give the first few values of ranks of $\Br(\ddB_{n})$.

\np
\begin{center}
\begin{tabular}{|c|c|c|c|c|c|}
     \hline
   $n$& 1&2 & 3 & 4 & 5 \\
   \hline
  $f(n)$&3&25 & 273 & 3801 & 66315 \\
  \hline
\end{tabular}
\end{center}

\np The case $n=2$ is discussed in \cite{CLY2010}, as $\ddB_2$ and $\ddC_2$
represent the same diagram.  Here we illustrate our results with the next
interesting case: $n=3$.
Let $$F=\{1, e_2, e_0, e_1e_0, e_0e_1, e_1e_0e_1, e_0e_2,
e_2r_1e_0e_1e_2\}.$$ We have the following decomposition of $\Br(\ddB_3)$
into $\Z[\delta^{\pm 1}]W(\ddB_3)$-submodules, where $W(\ddB_3)$ is the
submonoid of $\Br(\ddB_3)$ generated by $r_0$, $r_1$, $r_2$.
$$\Br(\ddB_3)=\bigoplus_{e\in F} \Z[\delta^{\pm 1}]W(\ddB_3)eW(\ddB_3).$$
Besides, the sizes of $W(\ddB_3)eW(\ddB_3)$ are $48$, $144$, $18$,
$18$, $18$, $9$, $9$, $9$ for $e\in F$ in the order they are listed.
This accounts for the rank of  $\Br(\ddB_3)$ being $273$.

The strategy of proof is as follows. The monomials in the canonical
generators of $\Br(\ddD_{n+1})$ are known to correspond to certain Brauer
diagrams with an additional decoration of order two by means of the
isomorphism $\psi:\Br(\ddD_{n+1})\rightarrow \BrD(\ddD_{n+1})$ introduced in
\cite{CGW2009}; here $\BrD(\ddD_{n+1})$ is an algebra linearly spanned by
the decorated classical Brauer diagrams, and the details are described in
Section \ref{sect:defns}.
The proof then consists of showing that the image of $\phi$ is the linear
span of the symmetric diagrams, which is free of rank
$f(n)$, and that $\Br(\ddB_n)$ is linearly spanned by at most $f(n)$
monomials. The latter is carried out by means of rewriting monomials to
normal forms, in such a way that each Brauer diagram corresponds to a unique
normal form.  This process leads to a basis which can be shown to be
cellular.

This paper has five sections.  Section \ref{sect:defns} gives the definition
of $\Br(\ddB_n)$ and some elementary properties of $\Br(\ddB_n)$ in
preparation of Section \ref{upperbound}.  We also recall results of Brauer
algebras of type $\ddD_{n+1}$ and present the surjectivity of the map $\phi$
of Theorem \ref{diagramimage} by combinatorial arguments in this
section. Section \ref{admissiblerootsets} discusses aspects of the root
system of type $\ddB_n$ that are used to identify monomials of $\Br(\ddB_n)$
in $r_0$, $\ldots$, $r_{n-1}$, $e_0$, $\ldots$, $e_{n-1}$; moreover a
pictorial description of some monomial images in $\BrD(\ddD_{n+1})$ is
presented.  In Section \ref{upperbound}, the rewriting of monomials of
$\Br(\ddB_n)$ is established, which leads to an upper bound on the rank of
$\Br(\ddB_n)$ and the main theorem is proved at the end of this section.  In
the last section, we establish that $\Br(\ddB_n)$ is cellular.

The idea of obtaining non-simply laced algebras from simply laced types has
been applied in \cite{Dieck2003} for Temperley-Lieb algebras of type $\ddB$
with generators $e_i$, in \cite{H1999} for the reduced BMW algebras of type
$\ddB$ and in \cite{Gra} for Hecke algebras of type $\ddB$.  In
\cite{ZhiChen}, Z.~Chen defines a Brauer algebra for each pseudo-reflection
group by use of a flat connection.  For types $\ddB$ and $\ddC$, it is
different from our algebra and has some intricate relations with our
algebras to be explained in further research.


\section{Definition and elementary properties} \label{sect:defns}
All our rings and algebras here will be unital (i.e., have an identity element)
and associative.

\begin{defn}\label{0.1}
Let $\Z[\delta^{\pm 1}]$ be the group ring over $\Z$  of the infinite cyclic group with generator $\delta$.
 For $n\in \N$, the \emph{Brauer algebra of type $\ddB_n$ over $\Z[\delta^{\pm 1}]$},
 denoted by $\Br(\ddB_n)$, is the
$\Z[\delta^{\pm 1}]$-algebra generated by $r_0$, $r_1,\dots, r_{n-1}$ and
$e_{0}$, $e_1,\dots, e_{n-1}$ subject to the following relations.

\begin{eqnarray}
r_{i}^{2}&=&1 \qquad \qquad\,\,\,\kern.02em \mbox{for}\,\mbox{any} \ i   \label{0.1.3}
\\
r_ie_i &= & e_ir_i \,=\, e_i \,\,\,\,\,\,\kern.05em \mbox{for}\,\mbox{any}\ i  \label{0.1.4}
\\
e_{i}^{2}&=&\delta e_{i} \qquad \qquad\kern.11em \mbox{for}\ i> 0    \label{0.1.5}
\\
e_{0}^{2}&=&\delta^2 e_{0}             \label{0.1.6}
\\
r_ir_j&=&r_jr_i \qquad \quad \mbox{for}\ i\perp j   \label{0.1.7}
\\
e_ir_j&=&r_je_i \qquad  \quad \kern-.03em \mbox{for}\ i\perp j     \label{0.1.8}
\\
e_ie_j&=&e_je_i \qquad \quad \kern-.06em \mbox{for}\ i\perp j      \label{0.1.9}
\\
r_ir_jr_i&=&r_jr_ir_j \qquad\,\kern-.04em \mbox{for}\ {i\sim j}\, \mbox{with}\ i, j > 0   \label{0.1.10}
\\
r_jr_ie_j&=&e_ie_j \quad \qquad \kern-.11em \mbox{for}\ i\sim j\ \mbox{with}\ i,j> 0               \label{0.1.13}
\\
r_ie_jr_i&=&r_je_ir_j \quad \quad \kern.06em \mbox{for}\ i\sim j\ \mbox{with}\ i, j> 0          \label{0.1.15}
\\
r_{1}r_0r_{1}r_{0}&=&r_0r_{1}r_0r_{1}                                       \label{0.1.11}
 \\
r_{0}r_1e_{0}&=&r_1e_{0}                                \label{0.1.14}
\\
  r_{0}e_1r_{0}e_1&=&e_1e_{0}e_1                                                         \label{0.1.19}
\\
(r_{0}r_1r_{0})e_1&=&e_1(r_{0}r_1r_{0})                                                            \label{0.1.20}
\\
e_{0}r_{1}e_{0}&=&\delta e_{0}                                                \label{0.1.12}
\\
e_{0}e_1e_{0}&=&\delta e_{0}                                                    \label{0.1.16}
\\
e_{0}r_1 r_{0}&=&e_{0}r_{1}                                                          \label{0.1.17}
\\
 e_{0}e_1r_{0}&=&e_{0}e_1                                                            \label{0.1.18}
\end{eqnarray}
Here $i\sim j$ means that $i$ and $j$ are adjacent in the Dynkin diagram
$\ddB_n$, and $\perp $ indicates that they are distinct and non-adjacent.
The submonoid of the multiplicative monoid of $\Br(\ddB_n)$ generated by
$\delta$, $\delta^{-1}$, $\{r_i\}_{i=0}^{n-1}$, and $\{e_i\}_{i=0}^{n-1}$ is
denoted by $\BrM(\ddB_n)$. It is the monoid of monomials in $\Br(\ddB_n)$
and will be called \emph{the Brauer monoid of type} $\ddB_n$.
\end{defn}

\np It is a direct consequence of the definition that the submonoid of
$\BrM(\ddB_n)$ generated by $\{r_i\}_{i=0}^{n-1}$ (the Weyl group
generators) is isomorphic to the Weyl group $W(\ddB_n)$ of type
$\ddB_n$. The algebra $\Br(\ddB_2)$ is isomorphic to $\Br(\ddC_2)$ defined
in \cite{CLY2010}, and the isomorphism is given by exchanging the indices
$0$ and $1$ of the Weyl group generators and of the Temperley-Lieb
generators.  As a consequence, Lemma 4.1 of \cite{CLY2010} applies in the
following sense.

\begin{lemma}In $\Br(\ddB_n)$, the following equalities hold.
\begin{eqnarray}
e_1e_0e_1&=&e_1r_{0}e_1  \label{4.1.1}
\\
r_{0}e_1 e_{0}&=&e_{1}e_{0}     \label{4.1.2}
\\
e_{0}r_{1}r_{0}e_{1}&=&e_{0}e_{1}       \label{4.1.3}
\\
r_1r_{0}e_{1}r_{0}&=&r_{0}e_{1}r_{0}r_1      \label{4.1.4}
\\
e_1r_{0}e_1r_0&=&e_1e_0e_1                \label{4.1.5}
\end{eqnarray}
\end{lemma}

\np We recall from \cite{CFW2008} the definition of a Brauer algebra of
simply laced Coxeter type $Q$.  In order to avoid confusion with the above
generators, the symbols of \cite{CFW2008} for the generators of $\Br(Q)$
have been capitalized.

\begin{defn}\label{1.1}
Let $Q$
be a simply laced Coxeter graph. The \emph{Brauer algebra of type $Q$ over
$R$ with loop parameter $\delta$}, denoted $\Br(Q)$, is the algebra
over $\Z[\delta^{\pm 1}]$ generated by $R_i$ and $E_i$, for each node $i$ of $Q$
subject to the following relations, where $\sim$ denotes
adjacency between nodes of $Q$ and $\perp$ non-adjacency for distinct nodes.

\begin{eqnarray}
R_{i}^{2}&=&1          \label{1.1.2} \\
E_{i}^{2}&=&\delta E_{i}   \label{1.1.4} \\
R_iE_i&=&E_iR_i \,=\, E_i     \label{1.1.3} \\
R_iR_j&=&R_jR_i \,\,\qquad\qquad\kern.05em \mbox{for}\, \it{i\perp j} \label{1.1.5} \\
E_iR_j&=&R_jE_i\,\,\qquad\qquad \kern.05em\mbox{for}\, \it{i\perp j}  \label{1.1.6} \\
E_iE_j&=&E_jE_i\,\, \qquad\qquad\kern.05em\mbox{for}\, \it{i\perp j}    \label{1.1.7} \\
R_iR_jR_i&=&R_jR_iR_j \,\,\quad\qquad\kern.05em \mbox{for}\, \it{i\sim j}  \label{1.1.8} \\
R_jR_iE_j&=&E_iE_j\,\,\qquad\qquad\kern.05em \mbox{for}\, \it{i\sim j}       \label{1.1.9} \\
R_iE_jR_i&=&R_jE_iR_j\,\,\quad\qquad\kern.05em \mbox{for}\, \it{i\sim j}     \label{1.1.10}
\end{eqnarray}
As before, we call $\Br(Q)$
\emph{the Brauer algebra of type $Q$}
and denote by $\BrM(Q)$
the submonoid of the multiplicative monoid of $\Br(Q)$ generated by
all $R_i$ and $E_i$, $\delta$, and $\delta^{-1}$.
 \end{defn}

\np
For each $Q$, the algebra $\Br(Q)$ is free over $\Z[\delta^{\pm 1}]$.
The classical Brauer algebra on $m+1$ strands arises when $Q=\ddA_m$.

\begin{rem}\label{additionalrelations}
It is straightforward to show that
the following relations hold in $\Br(Q)$ for all nodes $i$, $j$, $k$ with
$i\sim j \sim k$ and $i\perp k$.
\begin{eqnarray}
E_iR_jR_i&=&E_iE_j \label{3.1.1} \\ R_jE_iE_j &=& R_i E_j \label{3.1.2} \\
E_iR_jE_i &=& E_i \label{3.1.3} \\ E_jE_iR_j &=& E_j R_i \label{3.1.4} \\
E_iE_jE_i &=& E_i \label{3.1.5} \\ E_jE_iR_k E_j &=& E_jR_iE_k E_j
\label{3.1.6} \\ E_jR_iR_k E_j &=& E_jE_iE_k E_j \label{3.1.7}
\end{eqnarray}
\end{rem}

\np As in \cite[Remark 3.5]{CLY2010}, there is a natural anti-involution on
$\Br(\ddB_n)$. This anti-involution is denoted by the superscript
$\scriptstyle op$, so the map is denoted by $x\mapsto x^{\op}$ for any $x\in
\Br(\ddB_n)$.

\begin{prop} \label{prop:opp}
The identity map on $\{\delta, r_i, e_i \mid i = 0,\ldots,n-1\}$ extends to
the anti-involution $x\mapsto x^{\op}$ on the Brauer algebra $\Br(\ddB_n)$.
\end{prop}

Since $\Br(\ddB_2)\cong\Br(\ddC_2)$ and $\Br(\ddD_3)\cong \Br(\ddA_3)$, the
following corollary can be verified easily as in \cite{CLY2010}.

\begin{cor} \label{cor:2.6}
The map defined as $\phi$ on the generators of $\Br(\ddB_n)$ in Theorem
\ref{thm:main} extends to a unique algebra homomorphism $\phi$ on
$\Br(\ddB_n)$. Furthermore, the image of $\phi$ is contained in
$\SBr(\ddD_{n+1})$.
\end{cor}
\begin{proof}  The first claim can be verified  by checking defining relations under $\phi$. The second claim holds for the image of each generator of $\Br(\ddB_n)$
under $\phi$ is in  $\SBr(\ddD_{n+1})$.
\end{proof}
\np Since the subalgebra in $\Br(\ddD_{n+1})$ generated by $\{R_i,
E_i\}_{i=3}^{n+1}$ is isomorphic to $\Br(\ddA_{n-1})$, which can be found in
\cite{CGW2009}, or $\Br(\ddB_{n})/(e_0, r_0-1)\cong \Br(\ddA_{n-1})$, the
proposition below holds naturally.

\begin{prop} The subalgebra  generated by $\{r_i,\,e_i\}_{i=1}^{n-1}$ and $\delta$ in
$\Br(\ddB_n)$ is isomorphic to $\Br(\ddA_{n-1})$.
\end{prop}

\np Hence the formulas in Remark \ref{additionalrelations} still hold for
lower letters with nonzero indices.  The next two lemmas contain formulas
that will be applied in Section \ref{upperbound}.

For $ 2\le i\le n-1$, set $e_1^*=r_0e_1r_0$ and
$e_i^*=r_{i-1}r_ie_{i-1}^*r_ir_{i-1}$.

\begin{lemma}\label{ee*} For $i\in \{1,\ldots, n-1\}$,
\begin{eqnarray}
e_ie_i^{*}&=&e_ie_{i-1}\cdots e_1e_0e_1\cdots
e_{i-1}e_{i},\label{eieistar} 
\\
                r_0e_ie_{i}^*&=&e_ie_i^*.\label{r0eieistar}
\end{eqnarray}
\end{lemma}

\begin{proof} For $i=1$, we have $e_1e_1^{*}=e_1r_0e_1r_0\overset{(\ref{4.1.5})}{=}e_1e_0e_1.$
For $i>1$ induction, (\ref{0.1.15}), and (\ref{0.1.3})
give
$r_{i-1}r_{i}e_{i-1}r_ir_{i-1}=e_i$,
so
\begin{eqnarray*}
e_ie_i^*&=&r_{i-1}r_ie_{i-1}r_{i}r_{i-1} r_{i-1}r_{i}e_{i-1}^*r_{i}r_{i-1}
\overset{(\ref{0.1.3})}{=}r_{i-1}r_ie_{i-1}e_{i-1}^*r_{i}r_{i-1}\\
  &=&(r_{i-1}r_ie_{i-1})\cdots e_1e_0e_1\cdots  (e_{i-1}r_{i}r_{i-1})\\
         &\overset{(\ref{0.1.13})+(\ref{3.1.1})}{=}& e_ie_{i-1}\cdots e_1e_0e_1\cdots e_{i-1}e_{i}.
\end{eqnarray*}
This establishes (\ref{eieistar}).  Equality (\ref{r0eieistar})
follows from (\ref{eieistar}), (\ref{0.1.8}), and (\ref{4.1.2}).
\end{proof}

\np
Put
\begin{eqnarray}
\label{eq:g}
g&=&e_2r_1e_0e_1e_2.
\end{eqnarray}

\begin{lemma}\label{e2r1e0e1e2}
For $n\ge 3$,  the following equations hold in $\Br(\ddB_n)$,
\begin{eqnarray} g&=&g^{\op}\label{g=gop},\\
(r_1r_0r_1)e_2r_1e_0e_1&=&e_2r_1e_0e_1,\label{g=gop1}\\
(r_1r_0r_1)g&=&g,\label{g=gop2}\\
e_0g&=& \delta e_0e_2. \label{errreree3}
\end{eqnarray}
For $n\ge 4$, the following equations hold in $\Br(\ddB_n)$,
\begin{eqnarray}
(r_3r_2r_1r_0r_1r_2r_3)g&=&g,\label{g=gop3}\\
 e_0r_1r_2r_3g&=& \delta e_0e_1e_3r_2r_3, \label{errreree1}\\
 e_0r_1g&=& \delta e_0r_2r_1e_2, \label{errreree2}\\
 e_1r_2r_3g&=&e_1e_0r_1r_2r_3e_2.\label{errreree4}
 \end{eqnarray}
 \end{lemma}

\begin{proof} From
\begin{eqnarray*}
g&=&(e_2r_1)e_0e_1e_2\overset{(\ref{0.1.8})}{=} e_2e_1(r_2e_0)e_1e_2
\overset{(\ref{0.1.8})}{=}
       e_2e_1e_0(r_2e_1e_2)\overset{(\ref{3.1.2})}{=}e_2e_1e_0r_1e_2,
\end{eqnarray*}
it follows that $g$ is invariant under  opposition, and so (\ref{g=gop}).
Equality  (\ref{g=gop1}) follows from
\begin{eqnarray*}
r_1r_0(r_1e_2r_1)e_0e_1
&\overset{(\ref{0.1.15})}{=}& r_1(r_0r_2)e_1(r_2e_0)e_1
\overset{(\ref{0.1.8})+(\ref{0.1.9})}{=}r_1r_2(r_0e_1e_0)r_2e_1 \\
&\overset{(\ref{4.1.2})}{=}&   r_1r_2e_1(e_0r_2)e_1 
\overset{(\ref{0.1.9})}{=}(r_1r_2e_1r_2)e_0e_1\overset{(\ref{0.1.15})}{=}e_2r_1e_0e_1.
\end{eqnarray*}
Equality (\ref{g=gop2}) follows from (\ref{g=gop1}) by right multiplication
by $e_2$, and
Equality (\ref{g=gop3}) from
 \begin{eqnarray*}
r_3r_2r_1r_0r_1r_2r_3g&\overset{(\ref{0.1.13})}{=}&r_3r_2(r_1r_0r_1e_3)g
\overset{(\ref{0.1.8})}{=}r_3r_2e_3r_1r_0r_1g\\
&\overset{(\ref{g=gop2})}{=}&r_3r_2e_3g
\overset{(\ref{0.1.13})+(\ref{3.1.5})}{=}g.
\end{eqnarray*}

Formula (\ref{errreree3}) follows from
 \begin{eqnarray*}
 e_0g=(e_0e_2)r_1e_0e_1e_2&\overset{(\ref{0.1.9})}{=}& e_2(e_0r_1e_0)e_1e_2
\overset{(\ref{0.1.12})}{=}\delta e_2e_0e_1e_2\overset{(\ref{0.1.9})+(\ref{3.1.5})}{=}\delta e_0e_2,
 \end{eqnarray*}
Formula (\ref{errreree1})  from
 \begin{eqnarray*}
 e_0r_1r_2r_3g&=&e_0r_1(r_2r_3e_2)r_1e_0e_1e_2
\overset{(\ref{0.1.13})}{=}e_0(r_1e_3)e_2r_1e_0e_1e_2\\
&\overset{(\ref{0.1.8})}{=}&e_0e_3(r_1e_2r_1)e_0e_1e_2
\overset{(\ref{0.1.15})}{=}(e_0e_3r_2)e_1(r_2e_0)e_1e_2\\
&\overset{(\ref{0.1.8})+(\ref{0.1.9})}{=}&e_3r_2(e_0e_1e_0)r_2e_1e_2
\overset{(\ref{0.1.16})}{=}\delta (e_3r_2e_0r_2e_1)e_2\\
&\overset{(\ref{0.1.3})+(\ref{0.1.8})+(\ref{0.1.9})}{=}&
 \delta e_0e_1 e_3e_2
\overset{(\ref{3.1.1})}{=}\delta e_0e_1 e_3r_2r_3.
 \end{eqnarray*}
Formula (\ref{errreree2}) from
 \begin{eqnarray*}
e_0r_1g&=&e_0(r_1e_2r_1)e_0e_1e_2
\overset{(\ref{0.1.15})}{=}(e_0r_2)e_1(r_2e_0)e_1e_2
\overset{(\ref{0.1.8})}{=}r_2(e_0e_1e_0)r_2e_1e_2\\
&\overset{(\ref{0.1.16})}{=}&\delta (r_2e_0r_2)e_1e_2\overset{(\ref{0.1.3})+(\ref{0.1.8})}{=}\delta e_0 (e_1e_2)
\overset{(\ref{0.1.13})}{=} \delta e_0 r_2r_1e_2,
 \end{eqnarray*}
and Formula (\ref{errreree4}) from \begin{eqnarray*}
e_1r_2r_3g&=&e_1(r_2r_3e_2)r_1e_0e_1e_2
\overset{(\ref{0.1.13})}{=} (e_1 e_3)e_2r_1e_0e_1e_2\overset{(\ref{0.1.9})}{=} e_3 (e_1e_2r_1)e_0e_1e_2\\
  &\overset{(\ref{3.1.4})}{=}& (e_3 e_1r_2e_0)e_1e_2\overset{(\ref{0.1.8})+(\ref{0.1.9})}{=}e_1e_0e_3(r_2e_1e_2)
\overset{(\ref{3.1.2})}{=}e_1e_0(e_3r_1)e_2\\
&\overset{(\ref{0.1.8})}{=}&e_1e_0r_1(e_3e_2)\overset{(\ref{0.1.13})}{=}e_1e_0r_1r_2r_3e_2.
 \end{eqnarray*}
 \end{proof}

\np In order to give the diagram interpretation of monomials of
$\Br(\ddB_n)$, we recall the \emph{ Brauer diagram algebra of type}
$\ddD_{n+1}$ from \cite{CGW2009}.  Divide $2n+2$ points into two sets $\{1,
2,\ldots, n+1\}$ and $\{\hat{1}, \hat{2}, \ldots, \widehat{n+1}\}$ of points
in the (real) plane with each set on a horizontal line and point $i$ above
$\hat{i}$. An $n+1$-\emph{connector} is a partition on $2n+2$ points into
$n+1$ disjoint pairs. It is indicated in the plane by a (piecewise linear)
curve, called \emph{strand} from one point of the pair to the other. A
\emph{decorated} $n+1$-\emph{connector} is an $n+1$-connector in which an
even number of pairs are labeled $1$, and all other pairs are labeled by
$0$. A pair labeled $1$ will be called \emph{decorated}. The decoration of a
pair is represented by a black dot on the corresponding strand.  Denote
$T_{n+1}$ the set of all decorated $n+1$-connectors. Denote $T_{n+1}^0$ the
subset of $T_{n+1}$ of decorated $n+1$-connectors without decorations and
denote $T_{n+1}^=$ the subset of $T_{n+1}$ of decorated $n+1$-connectors
with at least one horizontal strand.

Let $H$ be the commutative monoid with presentation
$$H=\left<\delta^{\pm 1}, \xi, \theta\mid \xi^2=\delta^2, \xi
\theta=\delta\theta, \theta^2=\delta^2\theta\right> =\left<\delta^{\pm
1}\right>\{1, \xi, \theta \}.$$ A \emph{Brauer diagram} of type $\ddD_{n+1}$
is the scalar multiple of a decorated $n$-connector by an element of $H$
belonging to $\left<\delta^{\pm 1}\right>(T_{n+1}\cup \xi T_{n+1}^=\cup
\theta (T_{n+1}^{0}\cap T_{n+1}^=))$.  The \emph{Brauer diagram algebra of
type} $\ddD_{n+1}$, denoted $\BrD(\ddD_{n+1})$, is the
$\Z[\delta^{\pm1}]$-linear span of all Brauer diagrams of type $\ddD_{n+1}$
with multiplication laws defined in \cite[Definition 4.4]{CGW2009}. The
scalar $\xi\delta^{-1}$ appears in various products of $n+1$-connectors
described in \cite[Definition 4.4]{CGW2009} and two consecutive black dots
on a strand are removed.  The multiplication is an intricate variation of
the multiplication in classical Brauer diagrams, where the points of the
bottom of one connector are joined to the points of the top of the other
connector, so as to obtain a new connector. In this process, closed strands
appear which are turned into scalars by translating them into elements of
$H$ as indicated in Figure \ref{ch5closedloop}.

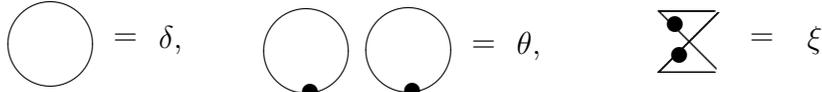
\begin{figure}[htbp]
\unitlength 1mm
\begin{picture}(26,15)(0,0)
\linethickness{0.3mm}
\put(10,9.41){\circle{11.18}}
\put(20,10){\makebox(0,0)[cc]{$=$}}
\put(26,10){\makebox(0,0)[cc]{$\delta$,}}
\end{picture}
$\qquad$
\unitlength 1mm
\begin{picture}(36.59,14.18)(0,0)
\put(30.59,9.18){\makebox(0,0)[cc]{$=$}}
\put(36.59,9.18){\makebox(0,0)[cc]{$\theta$,}}
\linethickness{0.3mm}
\put(20.59,8.59){\circle{11.18}}
\put(21.09,3.18){\circle*{2}}
\put(7,8.5){\circle{11.18}}
\put(7.5,3.09){\circle*{2}}
\end{picture}
$\qquad$
\unitlength 1mm
\begin{picture}(27,15)(0,0)
\linethickness{0.3mm}
\put(8.4,12){\circle*{2}}
\put(9,8){\circle*{2}}
\multiput(6.25,5.62)(0.08,0.08){100}{\line(0,1){0.09}}
\multiput(6.25,13.75)(0.08,-0.08){96}{\line(1,0){0.09}}
\put(20,10){\makebox(0,0)[cc]{$=$}}
\put(27,10){\makebox(0,0)[cc]{$\xi$}}
\multiput(6.25,13.75)(0.08,0.00){94}{\line(0,1){0.09}}
\multiput(6.25,5.62)(0.08,0.00){94}{\line(0,1){0.09}}
\end{picture}
\caption{The closed loops corresponding to the generators of $H$} 
\label{ch5closedloop}
\end{figure}

\begin{figure}[!htb]\label{psi}
\begin{center}
\includegraphics[width=.9\textwidth, height=.26\textheight]{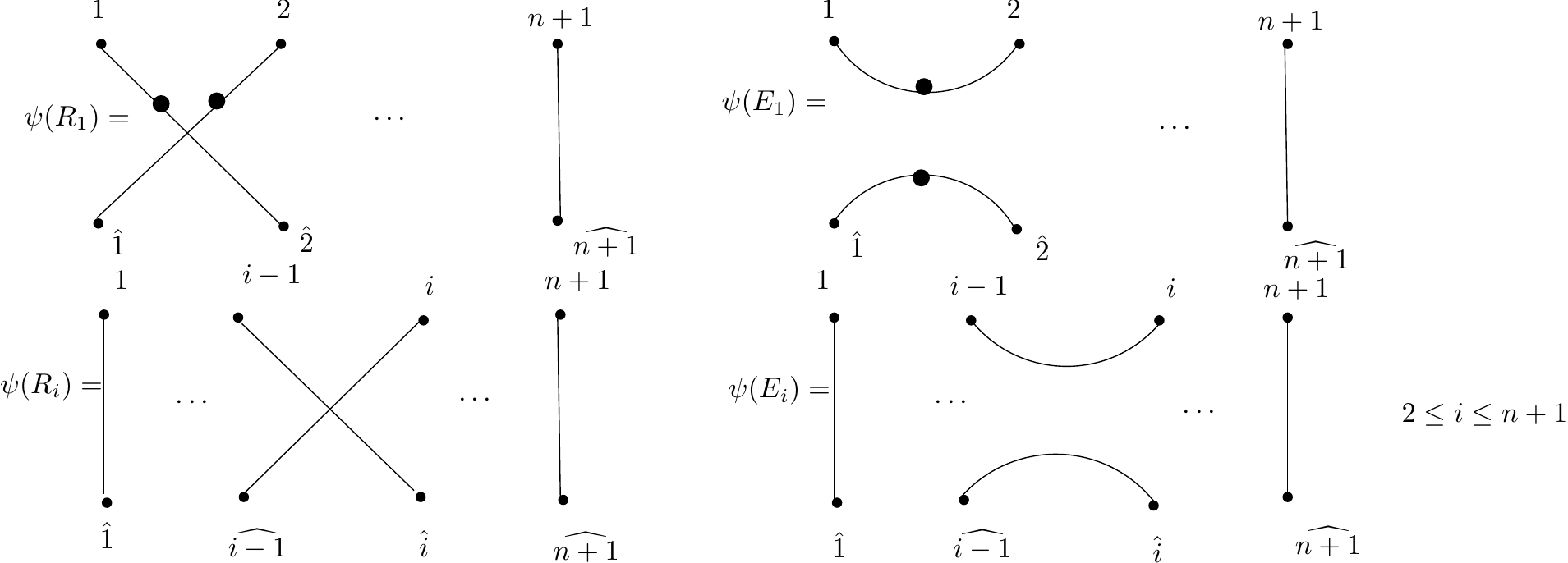}
\end{center}
\caption{The images of the generators of $\Br(\ddD_{n+1})$ under $\psi$}
\label{fig:1}
\end{figure}

In \cite{CGW2009}, the algebra $\BrD(\ddD_{n+1})$ is proved to be isomorphic
to $\Br(\ddD_{n+1})$ by means of the isomorphism
$\psi:\Br(\ddD_{n+1})\mapsto \BrD(\ddD_{n+1}) $ defined on generators as
in Figure \ref{fig:1}.  It is free over $\Z[\delta^{\pm 1}]$ with basis
$T_{n+1}\cup \xi T_{n+1}^=\cup \theta (T_{n+1}^{0}\cap T_{n+1}^=)$.

\begin{Notation}
\rm Write $T^{|}_{n+1}$ for the subset of $T_{n+1}$ consisting of all
$n+1$-connectors with a fixed strand from $1$ to $\hat{1}$, and set
$T^{|=}_{n+1}=T^|_{n+1}\cap T_{n+1}^=$.
It is readily checked that the union of $\delta^{\Z}T^{|}_{n+1}$ ,
$\delta^{\Z}\xi T^{|=}_{n+1}$, and $\delta^{\Z} \theta (T^{=}_{n+1}\cap
T^0_{n+1})$ is a submonoid of $\BrD(\ddD_{n+1})$; we denote it by
$\BrMD(\ddB_n)$ and the corresponding algebra over $\Z[\delta^{\pm 1}]$ by
$\BrD(\ddB_n)$.
\end{Notation}

\np The images of the generators of $\Br(\ddB_n)$ under $\psi\phi$ in
$\BrD(\ddD_{n+1})$ lie in $\BrD(\ddB_n)$; they are indicated in Figure
\ref{fig:5}.

\begin{figure}[!htb]
\begin{center}
\includegraphics[width=.7\textwidth, height=.26\textheight]{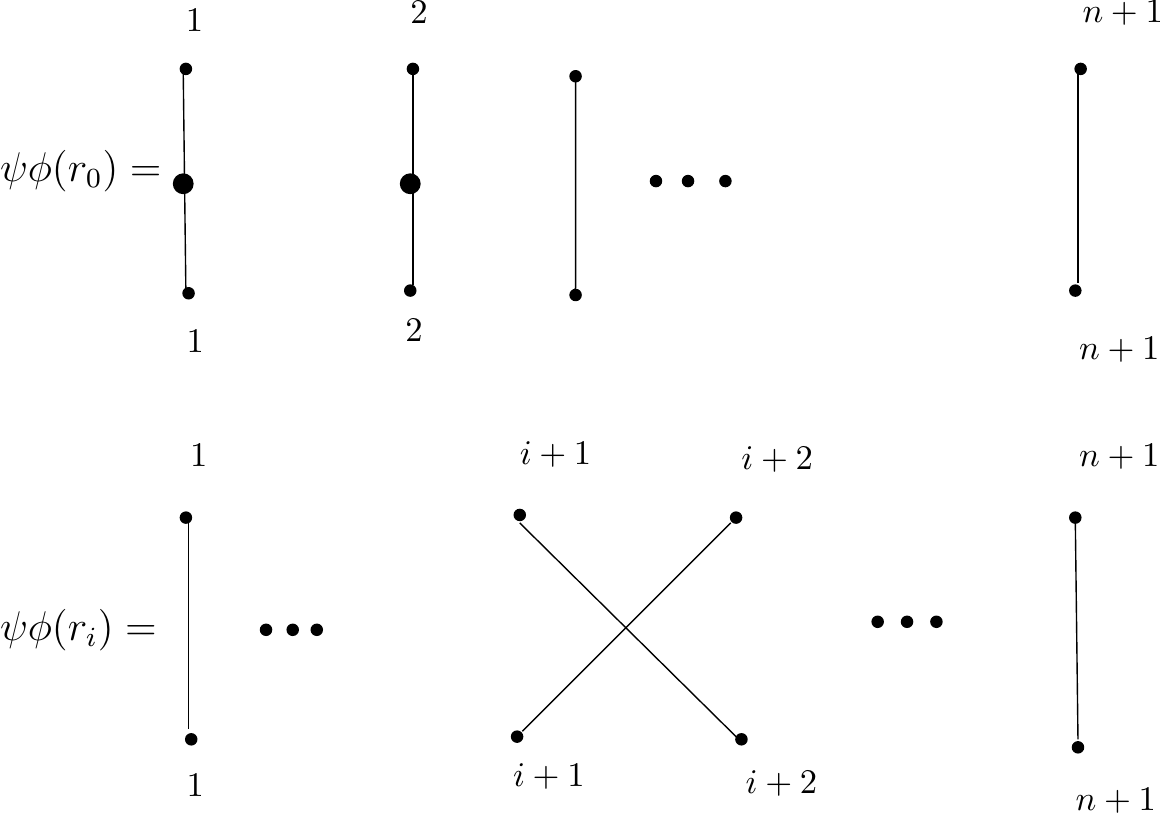}
\end{center}
\label{fig:4}
\end{figure}
\begin{figure}[!htb]
\begin{center}
\includegraphics[width=.7\textwidth, height=.26\textheight]{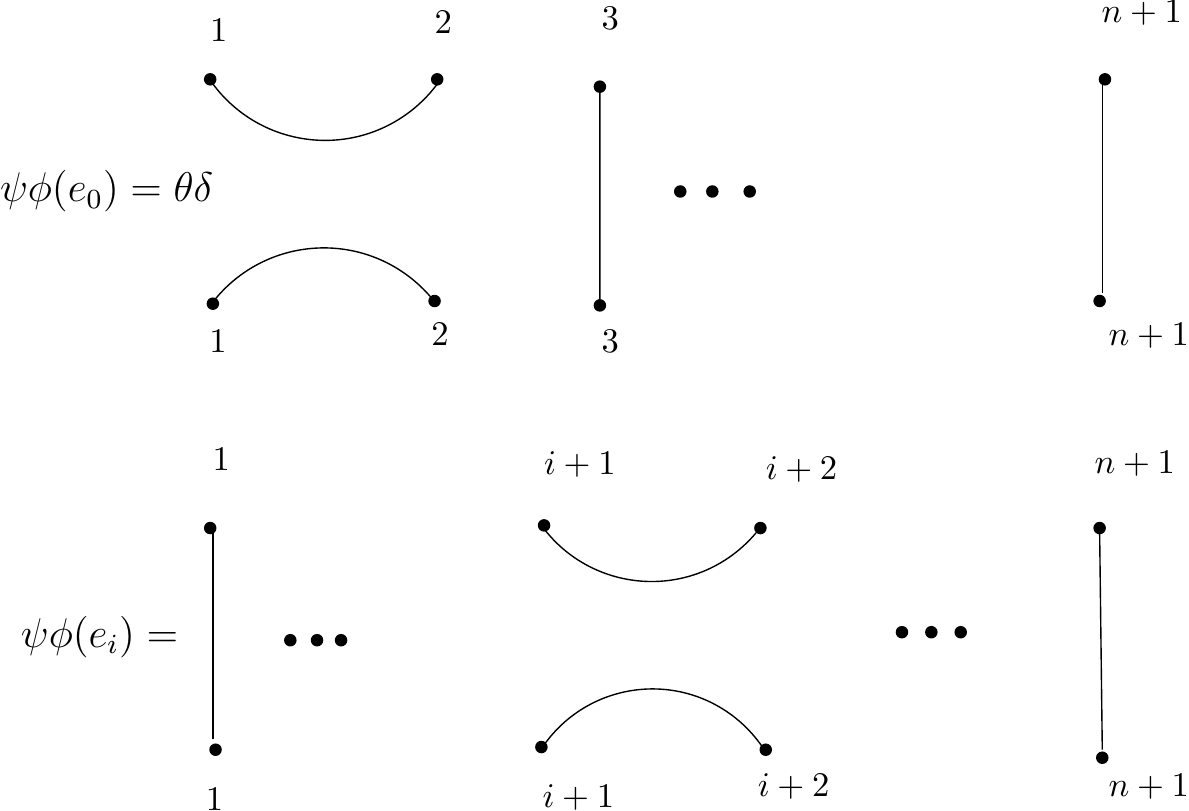}
\end{center}
\caption{The images under $\psi\phi$ of the generators of $\Br(\ddB_n)$}
\label{fig:5}
\end{figure}

\begin{thm} \label{diagramimage}
For $\phi$ and $\psi$, the following holds.
\begin{enumerate}[(i)]
\item The restriction of $\psi$ to $\phi(\Br(\ddB_n))$ is an isomorphism onto $\BrD(\ddB_n)$.
\item The image of $\phi$ coincides with $\SBr(\ddD_{n+1})$.
\item
The $\Z[\delta^{\pm 1}]$-algebras $\SBr(\ddD_{n+1})$ and $\BrD(\ddB_n)$ are
free of rank $f(n)$.
\end{enumerate}
\end{thm}

\np These assertions imply the commutativity of the following diagram.

\begin{center}
\phantom{longlong}
\xymatrix{
\Br(\ddB_n)\ar[dr]^{\phi}\ar[r]^{}&\SBr(\ddD_{n+1})\ar@{^{(}->}[d]\ar[r]^{}_{\cong}&\BrD(\ddB_n)\ar@{^{(}->}[d]\\
                      &\Br(\ddD_{n+1})\ar[r] ^\psi_{\cong}&\BrD(\ddD_{n+1}) }
\end{center}

\nl
\begin{proof}
(i).  All of the images of the generators $\{r_i,\, e_i\}_{i=0}^{n-1}$ under
$\psi\phi$ are in $\BrD(\ddB_n)$.  Therefore, the assertion is
equivalent to the statement that all elements in $\BrMD(\ddB_n)$ can be
written as products of $\{\psi\phi(r_i),\,\psi\phi (e_i)\}_{i=0}^{n-1}$ up
to some powers of $\delta$. Before we start to verify this fact, we observe
that the two special elements $K_i$ and $E_{i,j}$ in $\BrD(\ddB_n)$ of
Figure \ref{KiEij} are both in $\psi\phi(\Br(\ddB_n))$.

\begin{figure}[!htb]\label{KiEij}
\begin{center}
\includegraphics[width=.9\textwidth, height=.17\textheight]{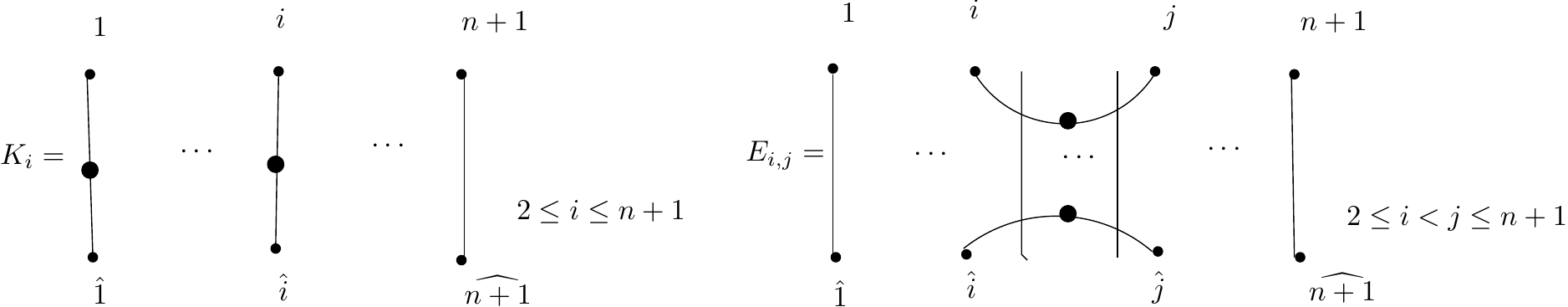}
\end{center}
\caption{$K_i$ and $E_{i,j}$}.
\end{figure}

This can be verified by induction on $i$ and $j$ (alternatively,
they are $\psi\phi(r_{\epsilon_{i}})$ and
$\psi\phi(e_{\alpha_j+\alpha_i})$ in the notation of section
\ref{admissiblerootsets}).

Let $a\in\BrMD(\ddB_n)$. We will show $a\in\psi\phi(\Br(\ddB_n))$.   For
$n=2$, $3$, the requires assertion can be checked by use of diagrams.  We
proceed by induction on $n$ and let $n>3$.

First assume $a\in \delta^{\Z}T^{|}_{n+1}\cup \delta^{\Z}\xi
T^{|=}_{n+1}$. Notice that $a$ has a vertical strand from $1$ to
$\hat{1}$. If $a$ has another vertical strand, say between $i$ and $\hat{j}$
with $1<i,j\leq n+1$, then, by multiplying $a$ by a suitable element of
$\left<\psi\phi(r_i)\right>_{i=1}^{n-1}$ (which is isomorphic to
$W(\ddA_{n-1})$) at the left and by
another element of it at the right, we can move this vertical strand so that it
connects $n+1$ and $\widehat{n+1}$. If it is decorated then we apply
$K_{n+1}$ (an element in the image of $\psi\phi$) to remove the decoration
on this strand; as a consequence, we may restrict ourselves to
$a\in\BrD(\ddB_{n-1})$.  But then induction applies and gives $a\in
\psi\phi(\Br(\ddB_n))$.

If $a$ has only one vertical strand, then the strands from $n+1$ and
$\widehat{n+1}$ are horizontal. As above, we multiply by two suitable
elements of $\left<\psi\phi(r_i)\right>_{i=1}^{n-1}$, which is isomorphic to
$W(\ddA_{n-1})$, to move the top horizontal strand to the strand connecting
$n$ and $n+1$ and the bottom horizontal strand to one connecting $\hat{n}$
and $\widehat{n+1}$. Next we apply $K_{n+1}$ to remove possible decorations
on the two strands. As a result, $a\in \Br(\ddB_{n-2})$ and so $a\in
\psi\phi(\Br(\ddB_n))$ by the induction hypothesis.

The case $a\in \delta^{\Z} \theta (T^{=}_{n+1}\cap T^0_{n+1})$ remains.
We distinguish five possible cases for $a$ by the strands with
ends $1$ and $\hat{1}$.

\begin{itemize}
\item $M^{(2)}$ is the subset of $T^{=}_{n+1}\cap T^0_{n+1}$ of all diagrams
with a fixed
vertical strand between $1$ and $\hat{1}$,
\item $M^{(3)}$ is the subset of $T^{=}_{n+1}\cap T^0_{n+1}$ of all
diagrams with two different horizontal strands with ends $1$ and $\hat{1}$,
\item $M^{(4)}$ is the subset of $T^{=}_{n+1}\cap T^0_{n+1}$ of all
 diagrams
 with two different vertical
strands with ends $1$ and $\hat{1}$, 
\item $M^{(5)}$ is the subset of $T^{=}_{n+1}\cap T^0_{n+1}$ of all diagrams
with a horizontal strand starting from $1$ and a vertical strand from
$\hat{1}$, and
\item  $M^{(6)}$ is the subset of $T^{=}_{n+1}\cap T^0_{n+1}$ of all
  diagrams with a
horizontal strand starting from $\hat{1}$ and vertical strands from $1$.
\end{itemize}

If $a\in \delta^{\Z}\theta M^{(2)}$, the diagram part can be written as a
scalar multiple of the
image of some element $b\in \delta^{\Z}
\psi\phi\left<r_i,\, e_{i}\right>_{i=1}^{n-1}$,
so $a = \delta^k\theta \psi\phi(b)$ for some $k\in\Z$. If there is a
horizontal strand at the top between $i$ and $j$, where $1<i<j\leq n+1$, then
$a=\delta^kE_{i,j}\psi\phi(b)$ for some $k\in \Z$, and we are done as 
$E_{i,j}$ lies in the image of $\psi\phi$.

If $a\in \delta^{\Z}M^{(5)}$ (or $M^{(6)}$, $M^{(4)}$, $M^{(3)}$,
respectively), then, by multiplying by suitable elements in $\psi\phi
\left<r_i\right>_{i=1}^{n-1}$ (which is isomorphic to $W(\ddA_{n-1})$) at
both sides of $a$, we can achieve that the strands between
$\{i,\hat{i}\}_{i=1}^3$ are as in the diagram of $\psi\phi(e_0e_1)$ (or
$\{i,\hat{i}\}_{i=1}^3$ as in $\psi\phi(e_1e_0)$, $\{i,\hat{i}\}_{i=1}^4$ as
in $\psi\phi(e_2r_1e_0e_1e_2)$, $\{i,\hat{i}\}_{i=1}^2$ as in
$\psi\phi(e_0)$, respectively).  Up to the leftmost $3$ or $4$ strands, the
resulting diagram can be considered as an element of $\delta^{\Z}\psi\phi
\left<r_i,\, e_{i}\right>_{i=3}^{n-1}$ (or $\delta^{\Z} \psi\phi\left<r_i,\,
e_{i}\right>_{i=3}^{n-1}$, $\delta^{\Z}\psi\phi\left<r_i,\,
e_{i}\right>_{i=4}^{n-1}$, $\delta^{\Z} \psi\phi\left<r_i,\,
e_{i}\right>_{i=2}^{n-1}$, respectively) which is isomorphic to
$\BrM(\ddA_{j})$ for $j = n-3$ (respectively, $j = n-3,n-4,n-2$). Therefore,
the first claim (i) holds.

\nl(ii).  It follows from (i) that
$$\psi^{-1}(\BrD(\ddB_n))=\phi(\Br(\ddB_n))\subseteq \SBr(\ddD_{n+1}).$$
Therefore, it suffices to prove
$\SBrM(\ddD_{n+1})\subseteq\phi(\BrM(\ddB_n))$, or, equivalently,
$$\left(\BrMD(\ddD_{n+1})\setminus \BrMD(\ddB_n)\right)\cap
\psi(\SBrM(\ddD_{n+1}))=\emptyset.$$ We find that $\BrMD(\ddD_{n+1})\setminus
\BrMD(\ddB_n) $ consists of $\delta^{\Z}M^{(i)}\cup \xi\delta^{\Z} M^{(i)}$,
for $i=3,\ldots, 6$.  By an argument analogous to the above and subsequently
multiplying by $K_i$ to remove decorations on all strands except the
leftmost $3$ or $4$ strands, we can reduce the verification to a case where
$n=2,3$, and so finish by induction.

\nl(iii).  By (i) and (ii), the algebras $\BrD(\ddB_n)$ and
$\SBr(\ddD_{n+1})$ are isomorphic $\Z[\delta^{\pm1}]$-algebras.  The latter
is free (as stated above) and the definition of the former shows that its
rank is $|T^{|}_{n+1}|+|T^{|=}_{n+1}|+|T^{=}_{n}\cap T^0_{n+1}|$.  A simple
counting argument gives $|T^{|}_{n+1}|=2^{n}\cdot n!!$,
$|T^{|=}_{n+1}|=2^{n}\cdot n!!-n!$, and $|T^{=}_{n}\cap
T^0_{n+1}|=(n+1)!!-(n+1)!$.  We conclude that the rank of $\BrD(\ddB_n)$ is
equal to $f(n)$.
\end{proof}

\np
For $n=3$, the eight Brauer diagrams in $\psi\phi(F)$ of
$\BrD(\ddD_4)$ (\cite[Section 4]{CGW2009}) are depicted in Figure
\ref{imageofS}.
 
\begin{figure}[!htb]
\begin{center}
\includegraphics[width=.9\textwidth, height=.3\textheight]{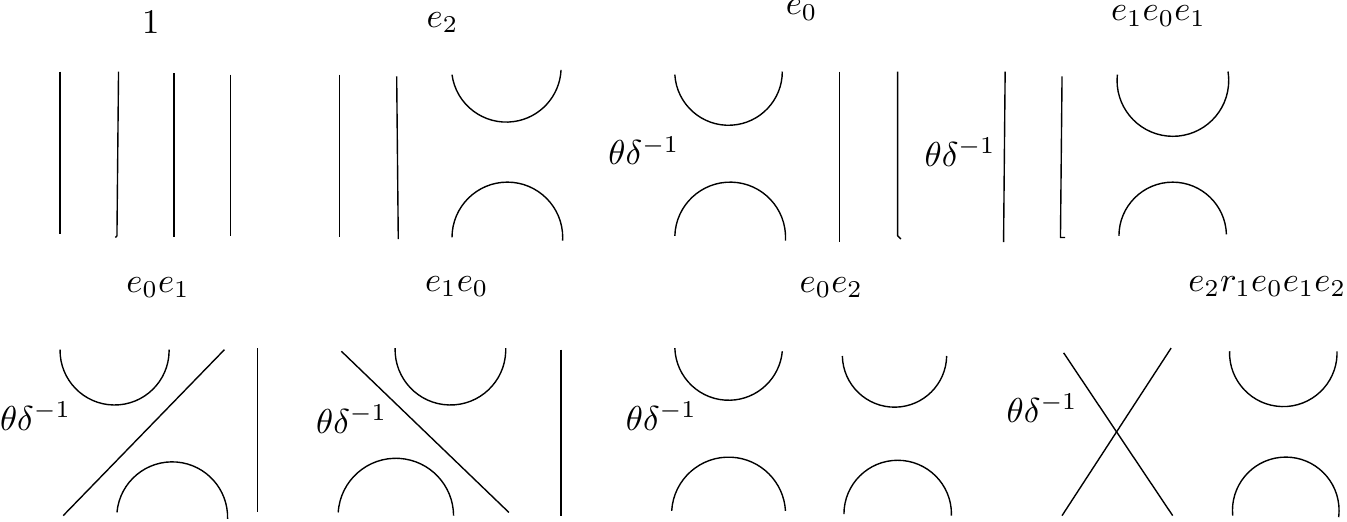}
\end{center}
\caption{The images of $F$ under $\psi\phi$}
\label{imageofS}
\end{figure}

\begin{section}{Admissible root sets and the monoid
action}\label{admissiblerootsets} In this section, we recall some facts
about the root system associated to the Brauer algebra of type $\ddD_{n+1}$,
introduce the definition of admissible root sets of type $\ddB_n$ and
describe some of its basic properties. Also by use of \cite{CLY2010}, we
give a monoid action of $\BrM(\ddB_n)$ on the admissible root sets.

\begin{defn} By $\Phi$ and $\Phi^+$, we denote the root system of type $\ddD_{n+1}$ and  its positive roots, 
and $\Phi^+$ can be realized by vectors $\epsilon_j\pm \epsilon_i$ with
$j>i$ in $\R^{n+1}$ where $\{\epsilon_i\}_{i=1}^{n+1}$ are the canonical
orthonormal basis.  The simple roots are $\alpha_1=\epsilon_1+\epsilon_2$
and $\alpha_i=\epsilon_i-\epsilon_{i-1}$ for $i=2, \ldots, n+1$. If $\alpha$
is a root of $\Phi^+$, then $\alpha^*$ denotes its \emph{orthogonal mate},
that is, $\alpha^*=\epsilon_j+\epsilon_i$ if $\alpha=\epsilon_j-\epsilon_i
\in \Phi^+$, and $(\alpha^*)^*=\alpha$.  When $n\ge 4$, this is the unique
positive root orthogonal to $\alpha$ and all other positive roots orthogonal
to $\alpha$ (see \cite[Definition 2.6]{CGW2008}).  The diagram automorphism
$\sigma$ on $W(\ddD_{n+1})$ is induced by a linear isomorphism $\sigma$ on
$\R^{n+1}$, where $\sigma$ is the orthogonal reflection with root
$\epsilon_1$. We define $\fp: \R^{n+1}\rightarrow \R^{n+1}$ by
$\fp(x)=(\sigma(x)+x)/2$.  The set $\Psi=\fp(\Phi)$ forms a root system of
the Weyl group of type $\ddB_n$ with
$\beta_0=\epsilon_2=(\alpha_1+\alpha_2)/2$ and
$\beta_i=\epsilon_{i+2}-\epsilon_{i+1}=\alpha_{i+2}$ for $i=1,\ldots, n-1$
as simple roots, and $\Psi^+=\fp(\Phi^+)$ consists of all positive roots of
the Weyl group of $W(\ddB_n)$.  A root $\beta \in \Psi$ is called a
\emph{long root} if $|\beta|=\sqrt{2}$, and called a \emph{short root} if
$|\beta|=1$.  The subset of $\Psi^+$ consisting of all short roots is
$\{\epsilon_i\}_{i=2}^{n+1}$.

A set of mutually orthogonal positive roots $A\subset\Phi^+$ is called
\emph{admissible} if, whenever $\gamma_1$, $\gamma_2$, $\gamma_3$ are
distinct roots in $A$ and there exists a root $\gamma\in \Phi $ for which
$|(\gamma,\gamma_i)|=1$ for all $i$, the positive root of $\pm R_\gamma
R_{\gamma_1}R_{\gamma_2}R_{\gamma_3}\gamma$ is also in $A$. An equivalent
definition states that either there is no orthogonal mate for any $\alpha
\in A$ or for each $\alpha\in A$ we have $\alpha^{*}\in A$.

A subset $B$ of $\Psi^+$ of mutually orthogonal roots is called
\emph{admissible}, if $\fp^{-1}(B)\cap \Phi\subset \Phi^+$ is admissible.
This is equivalent to $B$ being the image of some $\sigma$-invariant
admissible set of $\Phi^+$.

We write $\cA$ for the collection $\cA$ of admissible root subsets of
$\Phi^+$ and $\cA_\sigma$ for the set of all admissible subsets of $\Psi^+$
left invariant by $\sigma$. Its elements correspond to the
$\sigma$-invariant elements in $\cA$ via $B\mapsto \fp^{-1}(B)$.
\end{defn}

\begin{rem}  The mutually orthogonal root set $B_1=\{\epsilon_2, \epsilon_3\}$ is not admissible, for
 $\fp^{-1}(B_1)\cap \Phi=\{\alpha_1,
\,\alpha_2,\,\alpha_1+\alpha_3,\,\alpha_2+\alpha_3\}$ is not admissible.
Although $B_2=\{\alpha_1, \alpha_4\}$ is admissible in $\Phi^+$, the set
$\fp(B_2)$ is not admissible, as $\fp^{-1}\fp(B_2)=\{\alpha_1,\alpha_2,
\alpha_4\}$ is not admissible.
\end{rem}

\np By use of the description of admissible sets of type $\ddD$ in
\cite[Section 4]{CGW2008}, the admissible sets of type $\ddB$ can be
classified as indicated below.

\begin{prop}\label{classification}
The admissible root subsets of $\Psi^+$ can be divided into the following three  types.
\begin{itemize}
\item[(1)] The set consists of long roots and none of their orthogonal mates
are in this set.  It belongs to the $W(\ddB_n)$-orbit of
$Z_t=\{\beta_{n-1-2i}\mid 0\leq i< t\}$, for some $t$ with $0\leq
t<(n+1)/2$.
\item[(2)] The set consists of long roots and their orthogonal mates.  It
belongs to the $W(\ddB_n)$-orbit of $\tilde{Z}_t=\{\beta_{n-1-2i},
\beta_{n-1-2i}^{*}\mid 0\leq i< t\}$, for some $t$ with $1\leq t<(n+1)/2$.\\
\item[(3)] The set has only one short root, and for each long root that it
also contains, it also contains its orthogonal mate.  It belongs to the
$W(\ddB_n)$-orbit of $\bar{Z}_t=\{\beta_{n-1-2i}, \beta_{n-1-2i}^{*}\mid
0\leq i< t-1\}\cup \{\beta_0\}$, for some $t$ with $1\leq t\leq (n+1)/2$.
\end{itemize}
\end{prop}

\begin{lemma}\label{cardi}
 The cardinalities of the  $W(\ddB_{n})$-orbits  of $Z_{t}$, $\bar{Z}_{t}$, and $\tilde{Z}_{t}$ in the above   are
$2^t \binom{n}{2t} t!!$, $n \binom{n-1}{2t-2} (t-1)!!$, and $\binom{n}{2t} t!!$, respectively.
\end{lemma}

\begin{proof} The orbits of $Z_t$, $\bar{Z}_t$, and $\tilde{Z}_t$
correspond to, respectively,
\begin{enumerate}[(1)]
\item the diagrams with exactly $t$ decorated horizontal strands at the
top in $\BrD(\ddD_{n+1})$ none of which has end $1$, 
\item the diagrams with exactly $t$ horizontal
strands at the top without decoration one of which has end point $1$,
\item
the diagrams with exactly $t$ horizontal strands at the top without
decoration and none of which has end point $1$.
\end{enumerate}
Therefore, the sizes are easily seen to be as stated.
\end{proof}

\np Just as in \cite{CFW2008}, these results will be applied to compute the
rank of $\Br(\ddB_n)$.  The lemma below can be proved analogously to Lemma
5.7 of \cite{CLY2010}.

\begin{lemma}\label{lm:SimpleRootRels}
Let $i$ and $j$ be nodes of the Dynkin diagram $\ddB_n$.
If $w\in W(\ddB_n)$ satisfies $w\beta_i=\beta_j$, then
$we_iw^{-1}=e_j$.
\end{lemma}
Consider a positive root $\beta$ and a node $i$ of type $\ddB_n$. If there
exists $w\in W$ such that $w\beta_i=\beta$, then we can define the element
$e_{\beta}$ in $\BrM(\ddB_n)$ by
$e_{\beta}=we_iw^{-1}$.
The above lemma implies that $e_\beta$ is well defined. In general,
$we_{\beta}w^{-1}=e_{w\beta}$, for $w\in W(\ddC_{n+1})$ and $\beta$ a root
of $W(\ddB_n)$. Note that $e_{\beta}=e_{-\beta}$ in view of (\ref{0.1.4}).
We write $e_{\beta}^*=e_{\beta^*}$ and $e_i^*=e_{\beta_i^*}$ for
$\beta\in \Psi^+$ and $1\leq i\leq n-1$ in accordance with the definition before Lemma \ref{ee*}.

\begin{lemma} \label{2orthogonalroots}
If $\{\gamma_1,\,\gamma_2\}$ is a subset of some admissible root set, then
$$e_{\gamma_1}e_{\gamma_2}=e_{\gamma_2}e_{\gamma_1}.$$
\end{lemma}

\begin{proof}  
In view of Proposition \ref{classification}, the proof can be reduced to a
check in the following three cases

\centerline{$(\gamma_1,\gamma_2)=(\beta_0,\beta_2)$,
$(\beta_1, \beta_3)$, or
$(\beta_1,\beta_1^{*})$.}

\noindent
In the first two cases the lemma holds due to (\ref{0.1.9}). In the third
case it holds by (\ref{0.1.19}) and (\ref{4.1.5}).
\end{proof}

\np
Let $X\subset \Psi^+$ be  a subset of  some admissible root set.  Then we  define
\begin{eqnarray}\label{e}
e_{X}=\Pi_{\beta\in X}e_{\beta}.
\end{eqnarray}
In view of Lemma \ref{2orthogonalroots}, it is well defined.  Since the
intersection of admissible sets in $\cA$ is still an admissible set,
there is an admissible closure for mutually orthogonal subset of
$\Phi^+$, we can similarly give a definition for some subsets of $\Psi^+$ as
the following.

\begin{defn} 
Suppose that $X\subset \Psi^+$ is a mutually orthogonal root set. If $X$ is
a subset of some admissible root set, then the minimal admissible set
containing $X$ is called the \emph{admissible closure} of $X$, denoted by
$\overline{X}$.
\end{defn}

\begin{lemma} Let $X\subset \Psi^+$ be a mutually orthogonal root set and  $\overline X$ exists. Then
$$e_{\overline{X}}=\delta^{|\overline{X}\setminus X|}e_{X}.$$
\end{lemma}

\begin{proof}If $\overline{X}$ is in the $W(\ddB_n)$-orbit of $Z_{t}$, then it is trivial.
If $\overline{X}$ is in the $W(\ddB_n)$-orbit of $\bar{Z}_{t}$ or $\tilde{Z}_{t}$, then  $X$ can be transformed into  a subset of  $\bar{Z}_{t}$ or $\tilde{Z}_{t}$
by the action of  some element of $W(\ddB_n)$. Hence we just consider  subsets of $\bar{Z}_{t}$ and $\tilde{Z}_{t}$.
Now
\begin{eqnarray*}
e_0e_2e_2^{*}&=& e_0e_2r_1r_0(r_1e_2r_1)r_0r_1\overset{(\ref{0.1.15})}{=}e_0 e_2 r_1(r_0r_2)e_1r_2r_0r_1\\
&\overset{(\ref{0.1.7})}{=}&e_0 (e_2 r_1r_2)r_0e_1(r_2r_0)r_1\overset{(\ref{3.1.1})+(\ref{0.1.7})}{=}(e_0 e_2) (e_1r_0e_1r_0)r_2r_1\\
&\overset{(\ref{0.1.9})+(\ref{4.1.5})}{=}&e_2 (e_0 e_1e_0)e_1r_2r_1\overset{(\ref{0.1.16})}{=}\delta (e_2 e_0)e_1r_2r_1\\
&\overset{(\ref{0.1.9})}{=}&\delta e_0 (e_2e_1r_2r_1)\overset{(\ref{3.1.4})+(\ref{0.1.3})}{=}\delta e_0 e_2,
\end{eqnarray*}
\begin{eqnarray*}
e_{1}e_{1}^{*}e_{3}e_{3}^{*}
&\overset{(\ref{4.1.5})}{=}&e_1e_0(e_1e_3)r_2r_1r_0r_1r_2e_3r_2r_1r_0r_1r_2\\
&\overset{(\ref{0.1.9})}{=}&e_1e_0e_3(e_1r_2r_1)r_0r_1r_2e_3r_2r_1r_0r_1r_2\\
&\overset{(\ref{3.1.1})}{=}&e_1e_0e_3e_1(e_2r_0)r_1r_2e_3r_2r_1r_0r_1r_2\\
&\overset{(\ref{0.1.8})}{=}&e_1e_0(e_3e_1r_0)(e_2r_1r_2)e_3r_2r_1r_0r_1r_2\\
&\overset{(\ref{0.1.8})+(\ref{0.1.9})+(\ref{3.1.1})}{=}&e_1e_0e_1r_0e_3e_2(e_1e_3)r_2r_1r_0r_1r_2\\
&\overset{(\ref{0.1.18})}{=}&e_1e_0e_1e_3e_2e_3(e_1r_2r_1)r_0r_1r_2\\
&\overset{(\ref{3.1.1})}{=}&e_1e_0e_1(e_3e_2e_3)e_1e_2r_0r_1r_2\\
&\overset{(\ref{3.1.5})}{=}&e_1e_0e_1(e_3e_1)e_2r_0r_1r_2\overset{(\ref{0.1.9})}{=}e_1e_0(e_1e_1)e_3e_2r_0r_1r_2\\
&\overset{(\ref{0.1.5})}{=}&\delta e_1e_0e_1(e_3e_2r_0)r_1r_2\overset{(\ref{0.1.8})}{=}\delta (e_1e_0e_1r_0)e_3e_2r_1r_2\\
&\overset{(\ref{0.1.18})}{=}&\delta e_1e_0(e_1e_3)e_2r_1r_2\overset{(\ref{0.1.9})}{=}\delta e_1e_0e_3(e_1e_2r_1r_2)\\
&\overset{(\ref{3.1.4})}{=}&\delta e_1e_0e_3e_1(r_2r_2)\overset{(\ref{0.1.3})}{=}\delta e_1e_0e_3e_1\overset{(\ref{4.1.5})}{=}\delta e_{1}e_{1}^{*}e_{3}.
\end{eqnarray*}
This proves the lemma for $|X|=2$, $3$.
By applying induction on $|X|$, we see the lemma holds.
\end{proof}

\begin{rem}\label{rm:cA}
In \cite[Definition 3.2]{CFW2008}, an action of the Brauer monoid
$\BrM(\ddD_{n+1})$ on the collection $\cA$ is defined as follows.
The generators $\{R_{i}\}_{i=1}^{n+1}$ act
by the natural action of Coxeter group elements on its
root sets, where negative roots are negated so as to obtain positive roots,
the element $\delta$ acts as the identity,
and the action of $\{E_{i}\}_{i=1}^{n+1}$ is defined below.
\begin{equation}
E_i B :=\begin{cases}
B & \text{if}\ \alpha_i\in B, \\
\overline{B\cup \{\alpha_{i}\}} & \text{if}\ \alpha_i\perp B,\\
R_\beta R_i B & \text{if}\ \beta\in B\setminus \alpha_{i}^{\perp}.
\end{cases}
\end{equation}
In \cite[Section 4]{CGW2009}, a root $\epsilon_i-\epsilon_j$
($\epsilon_i+\epsilon_j$) with $1\leq j<i\leq n+1$ is represented as a
(decorated) horizontal strand from $i$ to $j$ at the top of the
diagram. This way, the above monoid action can be given a diagram
explanation in which admissible sets are represented by the set of
horizontal strands at the top of a diagram. 

\label{monoidaction}As in \cite[Proposition 5.6]{CLY2010},
there is a monoid action of $\BrM(\ddB_n)$ on $\cA_\sigma$ under the
composition of the above action of $\Br(\ddD_{n+1})$ and $\phi$. It gives an
action of $\Br(\ddB_n)$ on the collection of admissible subsets of $\Psi^+$.
This action also has a diagrammatic interpretation obtained form viewing the
admissible subsets of $\Psi^+$ as tops of symmetric diagrams by use of
$\fp^{-1}$.
\end{rem}
\end{section}

\begin{section}{Upper bound on the rank}
\label{upperbound}
In Theorem \ref{rewrittenforms} of this section, normal forms for elements
of the Brauer monoid $\BrM(\ddB_n)$ will be given, and in Corollary
\ref{rankatmost}, a spanning set for $\Br(\ddB_n)$ of size $f(n)$ will
be given. The rank of $\Br(\ddB_n)$ will be proved to be $f(n)$ as a
consequence.  These results will provide a proof of Theorem \ref{thm:main}.

The normal forms of monomials in $\Br(\ddB_n)$ will be parameterized by a
set $F$ of elements $f_t^{(i)}$ to be defined below, a general form being $a
= uf_t^{(i)}v$ for certain elements $u,v \in W(\ddB_n)$. In fact, $F$ is a
set of representatives for the $W(\ddB_n)$-orbits of the admissible sets
$a(\emptyset)$ and $a^{\op}(\emptyset)$ which appear in the guise of
horizontal strands at top and bottom, respectively, as discussed in Remark
\ref{rm:cA}. Recall $g=e_2r_1e_0e_1e_2$ from (\ref{eq:g}).

\begin{Notation}
\label{df:fti}
We define
\begin{eqnarray*}
f_t^{(1)}&:=&e_{Z_t}=\prod_{i=1}^{t}e_{n+1-2i},\, \quad\quad\quad\quad 0\leq t<(n+1)/2,\\
 f_t^{(2)}&:=&e_{\tilde{Z}_t}=\prod_{i=1}^{t}e_{n+1-2i}e_{n+1-2i}^*,\,\quad 1\leq t\leq n/2,\\
 f_t^{(3)}&:=&e_{\bar{Z}_t},\,\quad\quad\quad\quad\quad\quad\quad\quad\quad\quad  1\leq t\leq(n+1)/2,\\
  f_t^{(4)}&:=&gf_{t-1}^{(2)}, \quad\quad\quad\quad\quad\quad\quad\quad\quad 2\leq t\leq (n-1)/2,\\
  f_t^{(5)}&:=&e_0e_1f_{t-1}^{(2)}, \quad\quad\quad\quad\quad\quad\quad \quad 2\leq t\leq n/2,\\
   f_t^{(6)}&:=&e_1e_0f_{t-1}^{(2)}, \quad\quad\quad\quad\quad\quad\quad \quad 2\leq t\leq n/2,
  \end{eqnarray*}
and $f_1^{(4)}:=g$, $f_1^{(5)}:=e_0e_1$, $f_1^{(6)}:=e_1e_0$.  Furthermore,
we denote by $F$ the set of all elements $f_t^{(i)}$, and write $M =
\delta^{\Z}W(\ddB_n)FW(\ddB_n)$.
\end{Notation}

\np
The following statement is immediate from the definition of $M$.

\begin{lemma}\label{allr}
The set $M$ is closed under multiplication by  any element
from $W(\ddB_n)$.
\end{lemma}

\np Note that $f_t^{(3)}=e_0f_{t-1}^{(2)}$, for $2\leq t\leq [(n+1)/2]$.
The set $F$ is closed under the natural anti-involution $x\mapsto x^{\rm
op}$ of Proposition \ref{prop:opp}.

For $n=3$, we find $f_0^{(1)} = 1$, $f_1^{(1)} = e_2$, $f_1^{(3)} = e_0$,
$f_1^{(6)} = e_1e_0$, $f_1^{(5)} = e_0e_1$, $f_1^{(2)} = e_2e_2^* =
r_1r_2(e_1e_0e_1)r_2r_1$, $f_2^{(3)} =\delta e_0e_2$, and $f_1^{(4)} =g =
e_2r_1e_0e_1e_2$, which fits (up to conjugation by a Coxeter element) with
the eight elements of $F$ depicted in Figure \ref{imageofS}.

Part of Theorem \ref{rewrittenforms} states that each member of $M$ has a
unique decomposition, the other part states that $M$ coincides with
$\BrM(\ddB_n)$.  The first part, whose proof takes all of this section up to
the statement of the theorem, is devoted to giving a normal form for each
element of $M$, and involves a narrowing down of the possibilities for the
elements of $W(\ddB_n)$ in a normal form at both sides of $f_t^{(i)}$.  The
second part is carried out after the statements of Theorem
\ref{rewrittenforms}, where it is shown that $M$ is invariant under
multiplication by $r_i$ and $e_i$ for each $i\in\{0,\ldots,n-1\}$.  As $F =
F^{\op}$, it suffices to consider multiplication from
the left.  As $e_iuf_t^{(i)} = ue_{\beta}f_t^{(i)}$, where $\beta =
u^{-1}\beta_i$, it suffices to verify that $e_{\gamma}f_t^{(i)}$
belongs to $M$ for each $\gamma\in\Psi^+$. This is the content of Lemmas
\ref{shortef_t}--\ref{allee2r1e0e1e2f}.

\begin{defn}
\label{df:CD}
Let
$X$ be an admissible subset of $\Psi^+$. Recall the action of $W(\ddB_n)$ on all admissible subsets of $\Psi^+$ as
explained in Remark \ref{monoidaction}. The stabilizer of $X$ in
$W(\ddB_n)$ is denoted by $N(X)$.
We select a family $D_X$ of left coset representatives of $N(X)$
in $W(\ddB_n)$, so $|D_X|$ is the size of $W(\ddB_n)$-orbits of  $X$.  We simplify $D_{Z_t}$, $D_{\tilde{Z}_t}$,
$D_{\bar{Z}_t}$, $N_{Z_t}$,$N_{\tilde{Z}_t}$, $N_{\bar{Z}_t}$ to
$D_t^{(1)}$, $D_t^{(2)}$, $D_t^{(3)}$, $N_t^{(1)}$ $N_t^{(2)}$, $N_t^{(3)}$
respectively.

In order to identify the part of $W(\ddB_n)$ that commutes with $f_t^{(i)}$,
we introduce the following subgroups of $W(\ddB_n)$.
\begin{eqnarray*}
C_0^{(1)}&=&W(\ddB_n), \\
C_t^{(1)}&=&\left<r_{n-2}^*, r_0,r_1,\ldots, r_{n-2-2t}\right>,\,\qquad 1\leq t\leq n/2,\\
C_t^{(2)}&=&\left< r_1,r_2,\ldots, r_{n-2-2t}\right>,\,\qquad 1\leq t\leq n/2,\\
C_t^{(3)}&=&\left< r_2,r_3,\ldots, r_{n-2t}\right>,\,\qquad 1\leq t\leq (n+1)/2,\\
C_t^{(4)}&=&\left< r_4,r_5,\ldots, r_{n-2t}\right>,\,\qquad 1\leq t\leq (n-1)/2,\\
C_t^{(5)}=C_t^{(6)}&=&\left<r_3,r_4,\ldots, r_{n-2t}\right>,\,\qquad 1\leq t\leq n/2.
\end{eqnarray*}

For $0\leq t\leq [n/2]$, write
\begin{eqnarray*}
A_t^{(1)}&=&\left<r_{n-2i} r_{n+1-2i}  r_{n-1-2i} r_{n-2i} \right>_{i=1}^{t-1}\times\left<r_{n+1-2i}\right>_{i=1}^{t}, \\
W_t^{(1)}&=&\left< r_0,\,r_1,\ldots, r_{n-1-2t}\right>\times\left<r_{n+1-2i}^*\right>_{i=1}^{t}.
\end{eqnarray*}
 For  $1\leq t\leq [n/2]$, write
\begin{eqnarray*}
A_t^{(2)}&=&\left< r_{\epsilon_{n+2-2t}}, \,A_{t}^{(1)}\right>\times\left<r_{n+1-2i}^*\right>_{i=1}^{t}, \\
W_t^{(2)}&=&\left< r_0,\,r_1,\ldots, r_{n-1-2t}\right>.
\end{eqnarray*}
For $1\leq t\leq [(n+1)/2]$,  write
\begin{eqnarray*}
A_t^{(3)}&=&\left< r_{\epsilon_{n+4-2t}},\, r_{n-2i} r_{n+1-2i}
 r_{n-1-2i} r_{n-2i} \right>_{i=1}^{t-2}\times\left<r_{n+1-2i},r_{n+1-2i}^*\right>_{i=1}^{t-1}\times\left<r_0\right>,\\
W_t^{(3)}&=&\left< r_1r_0r_1,\,r_2,\,r_3,\ldots, r_{n+1-2t}\right>.
\end{eqnarray*}
For $1\leq t\leq (n-1)/2$, write
\begin{eqnarray*}
 A_t^{(4)}&=&\left< r_1r_0r_1,r_{\epsilon_{n+4-2t}-\epsilon_{4}}r_2r_{n+3-2t}r_{\epsilon_{n+4-2t}-\epsilon_{4}},A_{t-1}^{(1)},r_{2},r_2^*,\{r_{n+1-2i}, r_{n+1-2i}^*\}_{i=1}^{t-1},\,r_0\right>,\\
 W_t^{(4)}&=& \left<r_{\epsilon_5},  C_t^{(4)}\right>.
 \end{eqnarray*}

For $1\leq t\leq n/2$, write
\begin{eqnarray*}
 A_t^{(5)}&=&A_t^{(3)}\times\left<r_1r_0r_1\right>, \\
W_t^{(5)}&=&\left<r_2r_1r_0r_1r_2, \,r_3,\,r_4,\ldots, r_{n+1-2t} \right>.
\end{eqnarray*}
\end{defn}

\np
We first determine the structure of these subgroups.

\begin{lemma} \label{CAW} \begin{enumerate}[(i)]
\item For $1\leq t\leq [n/2]$, we have
\begin{eqnarray*}
C_{t}^{(1)}&\cong& W(\ddB_{n-2t})\times W(\ddA_{1}),\\
A_t^{(1)}&\cong& W(\ddA_{t-1})\times (W(\ddA_1))^t,\\
W_t^{(1)}&\cong& W(\ddB_{n-2t})(W(\ddA_{1}))^t.
\end{eqnarray*}
\item For $1\leq t\leq [n/2]$,  we have
\begin{eqnarray*}
C_t^{(2)}&\cong& W(\ddA_{n-1-2t}),\\
A_t^{(2)}&\cong& W(\ddB_{t})\times (W(\ddA_1))^{2t},\\
W_t^{(2)}&\cong& W(\ddB_{n-2t}).
 \end{eqnarray*}
\item For  $1\leq t\leq [(n+1)/2]$, we have
\begin{eqnarray*}
C_t^{(3)}&\cong& W(\ddA_{n-2t}),\\
A_t^{(3)}&\cong& W(\ddB_{t-1})\times (W(\ddA_1))^{2t-1}, \\
W_t^{(3)}&\cong& W(\ddB_{n+1-2t}).
\end{eqnarray*}

\item  For $1\leq t\leq (n-1)/2$,
  we have
 \begin{eqnarray*}
 A_t^{(4)}&\cong& W(\ddB_{t})\times W(\ddA_1)^{2t+1}, \\
  W_t^{(4)}&\cong& W(\ddB_{n-1-2t}),\\
  A_t^{(4)}\cap W_t^{(4)}&=&\{1\},\\
  C_t^{(4)}&\cong& W(\ddA_{n-2-2t}).
 \end{eqnarray*}
 Furthermore, the group $A_t^4$  normalizes $W_t^4$.
 \item For $1\leq t\leq n/2$, we have
 \begin{eqnarray*}
A_t^{(5)}&\cong& W(\ddB_{t-1})\times (W(\ddA_1))^{2t},\\
W_t^{(5)}&\cong& W(\ddB_{n-2t}),\\
A_t^{(5)}\cap W_t^{(5)}&=&\{1\},\\
C_t^{(5)}&\cong& W(\ddA_{n-1-2t}).
\end{eqnarray*}
Furthermore, $A_t^{(5)}$ normalizes $W_t^{(5)}$.
\end{enumerate}
\end{lemma}

\begin{proof}
With the diagram representation of Figure \ref{fig:4}, the
argument is analogous to \cite[Lemma 6.4]{CLY2010}.
\end{proof}

\begin{lemma} \label{normalizer}
For $i=1$, $2$, $3$, the subgroup $N_{t}^{(i)}$
 in $W(\ddB_n)$ is a  semiproduct of
$W_t^{(i)}$ and  $A_{t}^{(i)}$.
\end{lemma}
\begin{proof}
The semidirect group structure of these subgroups can be proved by use of
the diagram representation of Figure \ref{fig:4}.
The normalizer (stabilizer) claims follow from Lagrange's Theorem and  Lemma \ref{cardi}.
\end{proof}

\begin{defn}\label{df:Ns}
We need a few more subgroups and coset representatives.  For $i=4$, $5$, let
$N_t^{(i)}= A_t^{(i)} W_t^{(i)}$, and $D_t^{(i)}$ be its left coset
representatives in $W(\ddB_n)$.  The sets $\tilde{Z}_t$ and
$\tilde{Z}_{t-1}\cup\{\beta_1,\beta_1^*\}$ are conjugate by
$\tau=r_{\epsilon_{n+2-2t}-\epsilon_3}r_1r_{n+1-2t}r_{\epsilon_{n+2-2t}-\epsilon_3}$,
and $\tau C_t^{(2)}\tau^{-1}= C_t^{(6)}$.  At the same time we have $\tau
f^{(2)}_t\tau^{-1} =\delta e_{1}f^{(2)}_{t-1} $.  Let $N_t^{(6)}$ be the
stabilizer of $\tilde{Z}_{t-1}\cup\{\beta_1,\beta_1^*\}$ in $W(\ddB_n)$ and
$D_t^{(6)}$ be a set of left coset representatives of $N_t^{(6)}$ in
$W(\ddB_n)$. Let $ A_t^{(6)}=\tau A_t^{(2)}\tau^{-1}$ and $W_t^{(6)}=\tau
W_t^{(2)}\tau^{-1}$.

Finally, for $i=1$, $2$, $3$, $4$, let
$D_{t,L}^{(i)}=D_{t,R}^{(i)}=D_t^{(i)}$,
$D_{t,L}^{(5)}=D_{t,R}^{(6)}=D_t^{(5)}$, and
$D_{t,L}^{(6)}=D_{t,R}^{(5)}=D_t^{(6)}$.
\end{defn}

\begin{prop}\label{first3cases}
In $\Br(\ddB_n)$,  the following properties hold for all $i\in\{1,\ldots,6\}$.
\begin{enumerate}[(i)]
\item
For each $x\in N_t^{(i)}$  we have
$x f_t^{(i)}=f_t^{(i)}  x$.
\item
For each $a\in A_t^{(i)}$, we have
$a f_t^{(i)}=f_t^{(i)}$.
\item
For each $b\in W_t^{(i)}$
there exists some  $c\in C_t^{(i)}$,  such that
$bf_t^{(i)}=cf_t^{(i)}$.
\item
For each $c\in C_t^{(i)}$ we have
$cf_t^{(i)}=cf_t^{(i)}$.
\end{enumerate}
As a result,
\label{propM}
each monomial in $M$ can be written in the normal form
$$uf_t^{(i)}vw,$$
for some $u\in D_{t,L}^{(i)}$, $w\in (D_{t,R}^{(i)})^{\op}$,  $v\in C_t^{(i)}$, and $i\in\{1, \ldots, 6\}$.
\end{prop}

\begin{proof} (i). For $i\in\{1,2,3\}$,
this follows from Definition \ref{e} and Lemma \ref{normalizer}.  For
$i\in\{4,5,6\}$, observe that $N_t^{(i)}$ is the (semi-direct) product of
$W_t^{(i)}$ and $A_t^{(i)}$, so (i) will follow from (ii), (iii), and (iv).

\nl(ii). We use the following three equalities.
\begin{eqnarray}
r_ir_{i-1}r_{i+1}r_{i}e_{i-1}e_{i+1}&=&e_{i-1}e_{i+1},\quad \rm{for} \quad \, i>1 \label{rrrree}\\
r_0e_ie_i^*&=&e_ie_i^*,\\
r_ie_i&=&e_i. \label{re}
\end{eqnarray}
The first holds as
\begin{eqnarray*}
r_ir_{i-1}r_{i+1}r_{i}(e_{i-1}e_{i+1})&\overset{(\ref{0.1.9})}{=}&r_ir_{i-1}(r_{i+1}r_ie_{i+1})e_{i-1}\\
&\overset{(\ref{0.1.13})}{=}&(r_ir_{i-1}e_i)(e_{i+1}e_{i-1})\\
&\overset{(\ref{0.1.13})}{=}&(e_{i-1}e_ie_{i-1})e_{i+1}\\
&\overset{(\ref{3.1.5})}{=}&e_{i-1}e_{i+1},
 \end{eqnarray*}
the second equality is from Lemma \ref{ee*}, and  the third equality follows from  definition.
Now for the proof that the lemma holds for each generator, the formulas
(\ref{rrrree})--(\ref{re})  or their conjugations can cover all possible cases.

\nl(iii). The difference of generators of $W_t$ and $C_t$ is made up of
$\{r_{n+1-2i}^*\}_{i=2}^{t}$, which are conjugate to $r_{n-1}^*$ by some
elements in $A_{t}$, for example
$$r_{n-2}r_{n-3}r_{n-1}r_{n-2}\beta_{n-3}^*=\beta_{n-1}^*.$$ Therefore (iii)
for $f_t^{(1)}$ follows from (ii).  We can derive it for $f_t^{(2)}$ and
$f_t^{(3)}$ by applying Lemma \ref{ee*}.

For $i=4$, recall that $g=e_2r_1e_0e_1e_2$. By use of Lemma \ref{e2r1e0e1e2}
and
$$r_{\epsilon_{n+4-2t}-\epsilon_{4}} r_2r_{n+3-2t}r_{\epsilon_{n+4-2t}-\epsilon_{4}}e_2e_{n+3-2t}\overset{(\ref{rrrree})}{=}e_2e_{n+3-2t},$$
 the (ii) and (iii) about  $f_t^{(4)}$ can be obtained.

For $i=5$, the (ii) holds in view of $r_1r_0r_1 e_0=e_0$ and the argument of
 (ii) for $f_t^{(3)}$; the(iii) holds because of
$$r_2r_1r_0r_1r_2e_0\overset{(\ref{0.1.8})}{=}r_2(r_1r_0r_1e_0)r_2\overset{(\ref{0.1.14})}{=}r_2e_0r_2\overset{(\ref{0.1.8})}{=}e_0$$
  and (iii) for $f_t^{(3)}$.

When $i=6$, (ii) and (iii) hold naturally for (ii) and (iii) of $f_t^{(2)}$.

As for the final statement, Note that, if $w\in W$, we can write $w = vn$
with $v\in D_t^{(i)}$ and $n\in N_t^{(i)}$; but $n = ba$ with $b\in
W_t^{(i)}$ and $a\in A_t^{(i)}$; so, by (iii), $wf_t^{(i)} = vbf_t^{(i)} =
vf_t^{(i)}c$ for some $c\in C_t^{(i)}$.  Using the opposition involution of
Proposition \ref{prop:opp}, we can finish for $i = 1,2,3,4$ as
$\left(f_t^{(i)}\right)^{\op} = f_t^{(i)}$.  The other two cases can be
treated similarly, so we restrict ourselves to $i = 5$.

Applying the results obtained so far to
image under opposition of $f_t^{(5)}z$ for $z\in W$,
we find $u\in D_t^{(6)} = D_{t,R}^{(5)}$
and $d\in  C_t^{(6)}$
such that
$(f_t^{(5)}z)^{\op} = z^{-1} f_t^{(6)} =
u^{-1}f_t^{(6)}d^{-1}
$, so
$f_t^{(5)}z = df_t^{(5)} u$. As
$d\in  C_t^{(6)} =  C_t^{(5)} $ (see  Definition \ref{df:CD}),
this expression blends well with
$vf_t^{(5)}c$ for $wf_t^{(5)}$ to give the required normal form
for $wf_t^{(5)}z$.

\nl(iv). This follows from a straightforward check for each the generators
of $C_t^{(i)}$.
\end{proof}

\np We now come to the complete normal forms result by replacing $M$ in the
last statement of Proposition \ref{propM} with $\BrM(\ddB_n)$.

\begin{thm}\label{rewrittenforms}Up to powers of $\delta$,
each monomial in $\BrM(\ddB_n)$ can be written in  the  normal  form
$$uf_t^{(i)}vw,$$ for some $u\in D_{t,L}^{(i)}$, $w\in
(D_{t,R}^{(i)})^{\op}$, $v\in C_t^{(i)}$, and $f_t^{(i)}\in F$.
\end{thm}

\np In view of Proposition \ref{propM} and Lemma \ref{allr}, for the
proof of the theorem it remains to consider the products of the form
$e_{\beta}f_t^{(i)}$. This is done in Lemmas
\ref{shortef_t}---\ref{allee2r1e0e1e2f}.

\begin{lemma}\label{shortef_t}
For $f_t^{(1)}$ and $f_t^{(2)}$, the following statements hold.
\begin{enumerate}[(i)]
\item If $i<n-2t$, then there exist $r,s \in W(\ddB_n)$  such that
 \begin{eqnarray*}
 e_{\epsilon_{i+2}}f_t^{(1)} &=&\delta^{t} r f_{t+1}^{(3)} r^{-1}\qquad
\mbox{and}\qquad
 e_{\epsilon_{i+2}}f_t^{(2)} = s f_{t+1}^{(3)} s^{-1}.
 \end{eqnarray*}
\item
If $i\geq n-2t$, then there exists $s\in W(\ddB_n)$ such that
 \begin{eqnarray*}
 e_{\epsilon_{i+2}}f_t^{(1)} &=&\delta^{t-1} s  f_t^{(5)}s^{-1}\qquad
 \mbox{and}
\qquad
 e_{\epsilon_{i+2}}f_t^{(2)} =s  f_t^{(5)} s^{-1}.
 \end{eqnarray*}
\end{enumerate}
As a consequence, for each short root $\beta$ of $\Psi^+$ and each
$i\in\{1,2\}$, the monomial $e_\b f_t^{(i)}$ belongs to $M$.
\end{lemma}

\begin{proof} (i). We have $e_{\epsilon_{i+2}}f_t^{(1)}=e_{B}=\delta^{-t}e_{\overline{B}}$, where $B=\{\epsilon_{i+2}\}\cup Z_t$ and $\overline{B}=\{\epsilon_{i+2}\}\cup \tilde{Z}_t$ is on the $W(\ddB_n)$-orbit of $\bar{Z}_{t+1}$, hence the first equality.
The second equality holds by a similar argument. 

\nl(ii). First
consider the case $i=n-2t$.  For
$s=r_{\epsilon_{i+2}-\epsilon_{3}}r_ir_{i+1}
r_{\epsilon_{i+2}-\epsilon_{3}}$, we have
\begin{eqnarray*}
s\epsilon_{i+2}&=& \beta_0,\qquad
s \beta_{n+1-2t}=\beta_1,\qquad
s \beta_{n+1-2t}^*=\beta_1^*,
\end{eqnarray*}
and so
\begin{eqnarray*}
s (Z_{t}\setminus\{\beta_{n+1-2t}\})&=&(Z_{t}\setminus\{\beta_{n+1-2t}\}),\\
s (\tilde{Z}_{t}\setminus\{\beta_{n+1-2t}, \beta_{n+1-2t}^*\})&=&(\tilde{Z}_{t}\setminus\{\beta_{n+1-2t}, \beta_{n+1-2t}^*\}).\\
\end{eqnarray*}
Therefore
\begin{eqnarray*}
se_{\epsilon_{i+2}}f_{t}^{(1)}s^{-1}
&=&se_{\epsilon_{i+2}}s^{-1}se_{\beta_{n+1-2t}}s^{-1}sf_{t-1}^{(1)}s^{-1}
=e_0e_1f_{t-1}^{(1)}=\delta^{t-1}   f_t^{(5)},
\end{eqnarray*}
and similarly for $e_{\epsilon_{i+2}}f_{t}^{(2)}$ instead of
 $e_{\epsilon_{i+2}}f_{t}^{(1)}$.

Next, consider $i>n-2t$.  Now $\beta_{i+1}\in Z_t\subset \tilde{Z}_t$,
and $r_ir_{i-1}r_{i+1}r_{i}$ interchanges $\beta_{i+1}$ and
$\beta_{i-1}$ as well as $\beta_{i+1}^*$ and $\beta_{i-1}^*$;
moreover, it fixes all other elements of $Z_{t}$ and $\tilde{Z}_t$.
Hence $r_ir_{i-1}r_{i+1}r_{i}\epsilon_{i+2}=\epsilon_i$ and the lemma
holds by induction on $i$. For $\beta_i\in Z_t\subset \tilde{Z}_t$,
note that $r_i\{\epsilon_{i+2}\}=\{\epsilon_{i+1}\}$ and that $r_i$
keeps $Z_{t}$ and $\tilde{Z}_t$ invariant. It follows that, by
conjugation with $r_ir_{i-1}r_{i+1}r_{i}$ and $r_i$, the two
equalities are brought back to the cases for $i-2$ and $i-1$, respectively.

As for the final statement, note that each short root is of the form
$\eps_j$ for some $j\in\{2,\ldots,n+1\}$.
\end{proof}

\np
Let $W=W(\ddB_n)$.
\begin{lemma}\label{longef_t} If $\beta$  is a
long root in $\Psi^+$, then
  \begin{eqnarray*}
 e_{\beta}f_t^{(1)}&\in& \delta^{\Z}W f_t^{(1)}\cup\delta^{\Z} W f_{t+1}^{(1)}W\cup  \delta^{\Z}f_t^{(2)},\\
 e_{\beta}f_t^{(2)}
&\in& \delta^{\Z} W f_{t}^{(2)}\cup \delta^{\Z} W f_{t+1}^{(2)}W.
 \end{eqnarray*}
 \end{lemma}

 \begin{proof}Let's consider $e_{\beta}f_t^{(1)}$. First, if $\beta\in Z_t$, then $e_{\beta}f_t^{(1)}=\delta f_t^{(1)}$.
 Second, if $\beta$ is an orthogonal mate of any element in $Z_t$, we have
 $$e_{\beta}f_t^{(1)}= e_{\{\beta\}\cup Z_t}=\delta^{t-1}e_{\overline{\{\beta\}\cup Z_t}}=\delta^{t-1}f_t^{(2)},$$
 Third,  if  $\beta$ and $\beta'\in Z_{t}\cup
\tilde{Z}_t$ are two long roots and not orthogonal to each other. Then
$e_{\beta}e_{\beta'}\overset{(\ref{0.1.13})}{=}r_{\beta'}r_{\beta}
e_{\beta'}$, therefore $e_{\beta}f_t^{(1)}\in \delta^{\Z}W f_t^{(1)}$.
Fourth, $\beta$ is not in the above three cases, hence  $\{\beta\}\cup Z_t$  will be on the $W$-orbit of $Z_{t+1}$, therefore we have
$e_{\beta}f_t^{(1)}\in \delta^{\Z}W f_{t+1}^{(1)}W$.

The second claim of $e_{\beta}f_t^{(2)}$ holds  by a similar argument.
\end{proof}

\begin{lemma} \label{longee0f}
For each long root $\beta\in \Psi^+$, we have $e_\beta f_t^{(3)}\in M$.
\end{lemma}

\begin{proof} If $\beta$ is not orthogonal to all roots of $\tilde{Z}_{t-1}$, then
$e_{\beta}f_t^{(3)}=e_{\beta}f_{t-1}^{(2)}e_0$ by applying Lemma
\ref{longef_t} and Lemma \ref{allr}.

If $\beta$ is orthogonal to $\tilde{Z}_{t-1}$ and $\beta_0$, there exists a
$r\in N_t^{(3)}$, which does not move any element in $\tilde{Z}_{t-1}$ and
$\beta_0$ and $r\beta=\beta_{n+1-2t}$; hence
$e_{\beta}f_t^{(3)}=rf_{t+1}^{(3)} r^{-1}$.

If $\beta$ is orthogonal to
$\tilde{Z}_{t-1}$ and not orthogonal to $\beta_0$, then we  can find an element $r\in N_t^{(3)}$ such that $re_1r^{-1}=e_{\beta}$,
which implies that
$e_{\beta}f_t^{(3)}=r^{-1}e_1e_0f_{t-1}^{(2)}r\in Wf_{t}^{(6)}W$,  therefore our lemma holds.
\end{proof}

\begin{lemma} \label{shortee0f}
For each short root $\beta\in \Psi^+$, we have
$e_\beta f_t^{(3)}$,  $e_\beta f_t^{(5)}\in M$.
\label{shortee0e1f}
\end{lemma}

\begin{proof}
There is an index $i$ such that $\beta=r_{i}r_{i-1}\cdots r_{1} \beta_0$. Now
\begin{eqnarray*}
e_{\beta}e_0&=&r_{i}r_{i-1}\cdots r_{1} e_0r_{1}(r_{2}\cdots r_{i}e_0)
\overset{(\ref{0.1.8})}{=}r_{i}r_{i-1}\cdots r_{1} (e_0r_{1}e_0)r_{2}\cdots
r_{i}\\ &\overset{(\ref{0.1.12})}{=}&\delta r_{i}r_{i-1}\cdots r_{1}
e_0r_{2}\cdots r_{i} \overset{(\ref{0.1.8})}{=}\delta r_{i}r_{i-1}\cdots
r_{1} r_{2}\cdots r_{i}e_0\\
&\in&\delta W e_0,
\end{eqnarray*}
and the lemma follows as
both $f_t^{(3)}$ and  $f_t^{(5)}$ begin with $e_0$.
\end{proof}

\np
\begin{lemma}\label{longee0e1f}   Let $\beta\in \Psi^+$ be a long root.
\begin{enumerate}[(i)]
\item
If $\beta$ is not orthogonal to $\tilde{Z}_{t-1}$, then $e_{\beta}
f_t^{(5)}\in\delta^{\Z}Wf_{t}^{(5)}$.
\item
If $\beta=\epsilon_j\pm \epsilon_i$, with $2<i<j<n+4-2t$, then $e_{\beta} f_t^{(5)}\in \delta^{\Z}Wf_{t+1}^{(5)}W$.
\item
If $\beta=\epsilon_j\pm \epsilon_3$, with $3<j<n+4-2t$, then $e_{\beta} f_t^{(5)}\in \delta^{\Z}Wf_{t+1}^{(3)}W$.
\item
If $\beta=\epsilon_j\pm \epsilon_2$, with $2<j<n+4-2t$, then $e_{\beta} f_t^{(5)}\in \delta^{\Z}Wf_{t+1}^{(5)}W$.
\item
If $\beta=\beta_1$,  then $e_{\beta}
f_t^{(5)}=e_{1} f_t^{(5)}=rf_{t}^{(2)}r^{-1}$, for some $r\in W$.
\end{enumerate}
Thus for each $\beta\in \Psi^+$ the monomial $e_{\beta} f_t^{(5)}$ belongs
to $M$.
\end{lemma}

\begin{proof} As $r_{\epsilon_i}(\epsilon_j+\epsilon_i)=\epsilon_j-\epsilon_i$ and $\{r_{\epsilon_i}\}_{i=2}^{n+1}\subset N_t^{(6)}$,
 we only need consider $\beta=\epsilon_j-\epsilon_i$.
 Case (i)  can be checked easily.
 Case (v) holds as
  $$e_1f_{t}^{(5)}=e_1e_0e_1f_{t-1}^{(2)}\overset{(\ref{0.1.19})}{=}e_1e_1^{*}f_{t-1}^{(2)}=e_{\{\beta_1,\beta_1^*\}\cup \tilde{Z}_{t-1}},$$
  where $\{\beta_1,\beta_1^*\}\cup \tilde{Z}_{t-1}$ are on the $W$-orbits of
  $\tilde{Z}_t$.

After conjugation by suitable elements of $N_t^{(5)}$, we can restrict
ourselves to $\beta=\beta_{n+1-2t}$, $\beta_2$, $\beta_1+\beta_2$ for cases
(ii), (iii), and (iv), respectively.  Case (ii) can be proved easily. 
Case (iii) holds as
$e_2f_t^{(5)}=e_0e_2e_1f_{t-1}^{(2)}=e_0e_2r_1r_2f_{t-1}^{(2)}=e_0e_2f_{t-1}^{(2)}r_1r_2$,
and $e_0e_2f_{t-1}^{(2)}$ is conjugate to $e_0f_{t}^{(2)}$ by some element
of $W(\ddB_n)$. Case (iv) holds as
$e_{\beta_1+\beta_2}e_0e_1=r_1e_2r_1e_0e_1\overset{(\ref{3.1.4})}{=}r_1e_2r_1e_0e_1e_2r_1r_2=r_1gr_1r_2$.
\end{proof}

\np
\begin{lemma} \label{shortee1e0f}
For each short root $\beta\in \Psi^+$, we have $e_{\beta}f_t^{(6)}\in M$.
\end{lemma}

\begin{proof} If $\beta$ is equal to  $\beta_0$ or $\epsilon_3$, the lemma holds as
\begin{eqnarray*}
e_0f_t^{(6)}&=&e_0e_1e_0f_{t-1}^{(2)}\overset{(\ref{0.1.16})}{=}\delta e_0 f_{t-1}^{(2)}=\delta f_{t}^{(3)},\\
e_{\epsilon_3}f_t^{(6)}&=&r_1e_0r_1e_1e_0f_{t-1}^{(2)}=\delta r_1
e_0f_{t-1}^{(2)} =\delta r_1 f_{t}^{(3)}.
\end{eqnarray*}
If $\beta=\epsilon_{i+2}$ for $1<i<n$, then $e_{\beta}=r_i\cdots
r_2r_1e_0r_1r_2\cdots r_i$.  Hence
\begin{eqnarray*}
e_{\beta}f_t^{(6)}&=&r_i\cdots r_2r_1e_0r_1r_2\cdots r_i f_t^{(6)}\\
&\overset{(\ref{0.1.8})}{=}&r_i\cdots r_2r_1(e_0r_1r_2e_1e_0) r_3\cdots r_i f_{t-1}^{(2)}\\
&\overset{(\ref{0.1.13})}{=}&r_i\cdots r_2r_1(e_0e_2e_1e_0) r_3\cdots r_i f_{t-1}^{(2)}\\
&\overset{(\ref{0.1.9})+(\ref{0.1.16})}{=}&\delta r_i\cdots r_2r_1e_0 e_2 r_3\cdots r_i f_{t-1}^{(2)}\\
&=&\delta r_i\cdots r_2r_1 e_0 r_3\cdots r_i  (r_3\cdots r_i)^{-1} e_2 r_3\cdots r_i f_{t-1}^{(2)}\\
&\overset{(\ref{0.1.8})}{=}&\delta r_i\cdots r_2r_1 r_3\cdots r_i  e_0 (r_3\cdots r_i)^{-1} e_2 r_3\cdots r_i f_{t-1}^{(2)},\\
&\in& We_0e_{\epsilon_{i+2}-\epsilon_3} f_{t-1}^{(2)},\\
&\subseteq& We_0W( f_t^{(2)}\cup f_{t-1}^{(2)})W \quad\quad \quad\quad \mbox{by Lemma \ref{longef_t} }\\
&\subseteq& W( f_t^{(3)}\cup f_{t+1}^{(3)}\cup f_t^{(5)}\cup f_{t+1}^{(5)})W \quad\mbox{by Lemma \ref{shortef_t} }\\
&\subseteq&  M.
\end{eqnarray*}

\end{proof}

\begin{lemma}\label{longee1e0f}
If $\beta\in\Psi^+$ is a long root, then $e_{\beta}f_t^{(6)}$
belongs to $M$.
\end{lemma}

\begin{proof} If $\beta=\epsilon_{j+2}\pm \epsilon_{i+2}$ with $1<i<j<n$, then the lemma holds as
\begin{eqnarray*}
e_{\beta}e_1e_0f_{t-1}^{(2)}&=& e_1(e_0 (e_{\beta}f_{t-1}^{(2)}))\\
&\in& e_1e_0(\delta^{\Z}W f_{t-1}^{(2)}W\cup\delta^{\Z}W f_{t}^{(2)}W)\quad \quad \mbox{by Lemma \ref{longef_t}}\\
&\subseteq&\delta^{\Z}( e_1 W f_{t}^{(3)}W\hskip-.1cm \cup\hskip-.1cm W
f_{t-1}^{(5)}W\hskip-.1cm\cup\hskip-.1cm W
f_{t+1}^{(3)}W\hskip-.1cm\cup\hskip-.1cm
 W f_{t}^{(5)}W)\ \mbox{by Lemma \ref{shortef_t}}\\
&\subset& M \quad \quad \mbox{by Lemma \ref{longee0f} and Lemma \ref{longee0e1f}}.
\end{eqnarray*}
Otherwise, $\beta$ is not orthogonal to $\beta_1$, and so
$e_{\beta}e_1e_0f_{t-1}^{(2)}\overset{(\ref{0.1.13})}{=}r_1 r_{\beta}
e_1e_0f_{t-1}^{(2)}$ if $\beta\not\in \{\beta_1,\,\beta_1^*\}$, or $\delta
e_1e_0f_{t-1}^{(2)}$ if $\beta\in \{\beta_1,\,\beta_1^*\}$. Therefore, the
lemma holds.
 \end{proof}

\np Finally, we deal with $f_t^{(i)}$ for $i=4$.

\begin{lemma}\label{allee2r1e0e1e2f}
For each $\beta\in \Psi^+$, the monomial  $e_{\beta}f_t^{(4)}$
belongs to $M$.
 \end{lemma}

 \begin{proof}
First assume that $\beta$ is a short root. If $\beta=\epsilon_{i+2}$ with
$i>2$, then
\begin{eqnarray*}
e_{\beta}f_{t}^{(4)}&=&e_{\beta}gf_{t-1}^{(2)}=r_i\cdots r_4 r_3 r_2r_1e_0r_1r_2r_3\cdots r_i
gf_{t-1}^{(2)}\\ &\overset{(\ref{0.1.8})}{=}&r_i\cdots r_4 r_3
r_2r_1(e_0r_1r_2r_3 g) r_4\cdots r_if_{t-1}^{(2)}\\
&\overset{(\ref{errreree1})}{=}& \delta r_i\cdots r_4 r_3 r_2r_1(e_0(e_1
(e_3r_2\cdots r_i f_{t-1}^{(2)})))\\
&\in&\delta^{\Z} We_0f_{t'}^{(2)}W\qquad
\mbox{by Lemma \ref{longef_t}}\\
&\subseteq&M\qquad \mbox{by Lemma \ref{shortef_t}}
\end{eqnarray*}
If $\beta=\epsilon_4$, $\epsilon_3$, or $\beta_0$, the same argument as
above can be applied with (\ref{errreree2}) and (\ref{errreree3}) in Lemma
\ref{e2r1e0e1e2}.

Next assume $\beta$ is a long root.  If $\beta=\epsilon_{j+2}\pm
\epsilon_{i+2}$ with $2<i<j<n$,  then $e_{\beta}gf_{t-1}^{(2)}=g e_{\beta}
f_{t-1}^{(2)}$.  If $\beta$ is not orthogonal to $\tilde{Z}_{t-1}$ (see
Proposition \ref{classification}), then either $e_{\beta}
f_{t-1}^{(2)}=r_{s}r_{\beta}f_{t-1}^{(2)}$, for some $\beta_s\in
\tilde{Z}_{t-1}$ not orthogonal to $\beta$, or $e_{\beta}
f_{t-1}^{(2)}=\delta f_{t-1}^{(2)}$. Now the lemma holds as $
r_{s}r_{\beta}g=gr_{s}r_{\beta}$.  Otherwise,
$j<n+2-2t$, and we can find some element $r\in \tilde{N}_{t-1}^{(2)}$
 such
that $r\beta=\beta_{n+1-2t}$. Then $e_{\beta}
f_{t-1}^{(2)}=r^{-1}e_{n+1-2t}rf_{t-1}^{(2)}=r^{-1}e_{n+1-2t}f_{t-1}^{(2)}r=\delta^{-1}r^{-1}f_{t}^{(2)}r
$, for $r^{-1}$ commutes with $g$, hence the lemma holds by Lemma
\ref{allr}.

If $\beta=\epsilon_{j+2}\pm \epsilon_{i+2}$ with $0<i\leq 2 \leq j<n$, the
root $\beta$ is not orthogonal to $\beta_2$ or $\beta\in
\{\beta_2,\,\beta_2^{*}\}$, hence $e_\beta g=r_2r_{\beta}g$ or $\delta g$,
which implies that the lemma holds by Lemma \ref{allr}.

If
$\beta=\epsilon_{j+2}\pm \epsilon_2$, then $r_0\beta =\beta^*$ and $r_0\in
N_t^{(2)}$, so it suffices to consider
$\beta=\epsilon_{j+2}-\eps_2$. If $2<j$, then
 \begin{eqnarray*}
 e_{\beta}f_{t}^{(2)}&=& e_{\beta}gf_{t-1}^{(2)}=r_j\cdots r_4 r_3 r_2e_1r_2r_3\cdots r_j gf_{t-1}^{(2)}\\
 &=&r_j\cdots r_4 r_3 r_2(e_1r_2r_3 g) r_4\cdots r_jf_{t-1}^{(2)}\\
 &\overset{(\ref{errreree4})}{=}& r_j\cdots r_4 r_3
 r_2(e_1(e_0r_1r_2r_3(e_2r_4\cdots r_j f_{t-1}^{(2)})))\\
&\in&\delta^{\Z} W e_1e_0Wf_{t-1}^{(2)}W\qquad
\mbox{by Lemma \ref{longef_t}}\\
&\subseteq&\delta^{\Z} W e_1Wf_{t-1}^{(2)}W\qquad
\mbox{by Lemma \ref{shortef_t}}\\
&\subseteq&\delta^{\Z} Wf_{t-1}^{(2)}W,\qquad
\mbox{by Lemma \ref{longee1e0f}}
\end{eqnarray*}
and we are done.
The remaining two cases are $\beta=\beta_1$ and $\beta=\beta_1+\beta_2$.
These follow readily from $e_1e_2=r_2r_1e_2$ and $r_2e_1r_2e_2=r_1e_2$.
This proves the lemma.
\end{proof}

\np By means of Lemmas \ref{shortef_t}----\ref{allee2r1e0e1e2f}, we have
shown that $e_\beta f_t^{(i)}\in M$ for each $\beta\in\Psi^+$ and each
$i\in\{1,\ldots,6\}$, which suffices to complete the proof of Theorem
\ref{rewrittenforms}.

We proceed to give an upper bound for the rank of the Brauer algebra
$\Br(\ddB_n)$. By Theorem \ref{rewrittenforms}, the upper bound is given by
\begin{eqnarray}\label{eq:sum}
\sum_{i=1}^{6}|D_{t,L}^{(i)}||D_{t,R}^{(i)}||C_t^{(i)}|.
\end{eqnarray}

Table \ref{tab:1} lists the cardinalities of $D_{t,L}^{(i)}$,
$D_{t,R}^{(i)}$, and $C_t^{(i)}$.

\begin{table}[!htb]
\caption{Cardinalities of coset and centralizers}
\label{tab:1}
\begin{center}
\begin{tabular}{|c|c|c|c|}
\hline
  $D_{0}^{(1)}$ & $1$& $C_{0}^{(1)}$ & $2^{n} n!$\\
  \hline
  $D_{t}^{(1)}$ & $2^t \binom{n}{2t} t!!$& $C_{t}^{(1)}$ & $2^{n+1-2t}(n-2t)!$\\
   \hline
  $D_t^{(2)}$  & $\binom{n}{2t} t!!$&$C_t^{(2)}$  & $(n-2t)!$ \\
  \hline
  $D_t^{(3)}$& $n \binom{n-1}{2t-2} (t-1)!!$ &$C_t^{(3)}$& $(n+1-2t)!$ \\
  \hline
  $D_t^{(4)}$ & $n \binom{n-1}{2t} t!!$ &$C_t^{(4)}$ & $(n-1-2t)!$\\
 \hline
  $D_t^{(5)}$& $n(n-1)\binom{n-2}{2t-2} (t-1)!!$ &$C_t^{(5)}$& $(n-2t)!$ \\
  \hline
  $D_t^{(6)}$ & $\binom{n}{2t} t!!$&$C_t^{(6)}$& $(n-2t)!$ \\
  \hline
\end{tabular}
\end{center}
\end{table}
\begin{lemma}\label{numerical}
For $0<t\leq (n+1)/2$,
$$\sum_{i=2}^{6}|D_{t,L}^{(i)}||D_{t,R}^{(i)}||C_t^{(i)}|=\left(\binom{n+1}{2t}t!!\right)^2 (n+1-2t)!.$$
\end{lemma}

\begin{proof}
Recall our $M^{(i)}$, $i=2,\,\ldots,\,6$, in the proof of Theorem
\ref{diagramimage}, and let $M^{(i)}_t$ be the subset of diagrams $M^{(i)}$
with $t$ horizontal strands at the top of their diagrams.  These $M^{(i)}_t$
consist of all possible diagrams with $t$ horizontal strands in
$T_{n+1}^{=}\cap T_{n+1}^{0}$.
The count of classical Brauer diagrams (of \cite{Brauer1937}) related to the
Brauer monoid of type $\Br(\ddA_{n})$ with $t$ horizontal strands at the top
can be conducted as follows.  First choose $2t$ points at the top (bottom)
and make $t$ horizontal strands; the remaining $n+1-t$ vertical strands
correspond to the elements of the Coxeter group of type $W(\ddA_{n-t})$.
Therefore the right hand side of the equality is the number of all possible
diagrams in $T_{n+1}^{=}\cap T_{n+1}^{0}$ with $t$ horizontal strands at
the top, and so equals
$$\sum_{i=2}^{6}|M^{(i)}_t|=\left(\binom{n+1}{2t}t!!\right)^2 (n+1-2t)!.$$
We compute $|M^{(i)}_t|$ for $i=2,\ldots,6$.

For $i=2$, we just count as above with $n+1$ replaced by
$n$, so
$$|M^{(2)}_t|=\left(\binom{n}{2t}t!!\right)^2 (n-2t)!=
|D_t^{(2)}|^2|C_t^{(2)}| = |D_{t,L}^{(2)}||D_{t,R}^{(2)}||C_t^{(2)}|.$$ For
$i=5$, we first choose two points from the top $n+1$ points except $1$ for
the horizontal strand from $1$ and the vertical strand from $\hat{1}$, and
we choose $2(t-1)$ points at the top from the remaining $n+1-3$ points at
the top for $t-1$ horizontal strands and $2t$ points at the bottom $n+1$
points except $\hat{1}$ for $t$ horizontal strands, and then the vertical
strands between the remaining $n-2t$ points at the top and the remaining
$n-2t$ points at the bottom will be corresponding to elements of the Coxeter
group of type $W(\ddA_{n-2t-1})$. Therefore,
$$|M^{(5)}_t|=n(n-1)\binom{n-2}{2t-2}(t-1)!!\binom{n}{2t}t!! (n-2t)!$$
$$= |D_t^{(5)}||D_t^{(6)}||C_t^{(5)}| = |D_{t,L}^{(5)}||D_{t,R}^{(5)}||C_t^{(5)}|.$$
By reversing the top and bottom, we obtain a one to one correspondence
between $M^{(5)}_t$ and $M^{(6)}_t$; it follows that
$$|M^{(6)}_t|=|M^{(5)}_t|=|D_{t,L}^{(6)}||D_{t,R}^{(6)}||C_t^{(6)}|.$$
For $i = 4$, we first choose one point from the bottom (top) $n+1$ points
except $\hat{1}$ ($1$) for the vertical strand from $1$ ($\hat{1}$); the
remaining count of horizontal strands and other vertical stands is as
in the classical case after replacing $n+1$ by $n-1$; hence 
$$|M^{(4)}_t|=n^2\left(\binom{n-1}{2t}t!!\right)^2 (n-1-2t)!=
|D_t^{(4)}|^2|C_t^{(2)}| = |D_{t,L}^{(4)}||D_{t,R}^{(4)}||C_t^{(4)}|.$$ For
$i = 3$, we first choose one point at the top (bottom) distinct from $1$
($\hat{1}$) for the horizontal strand from $1$ ($\hat{1}$); the
remaining count of other horizontal strands and vertical stands is as
in the classical case after replacing $n+1$ by $n-1$ and $t$ by $t-1$; it
follows that
\begin{eqnarray*}
|M^{(3)}_t|&=&n^2\left(\binom{n-1}{2t-2}(t-1)!!\right)^2 (n+1-2t)!=
|D_t^{(3)}|^2|C_t^{(3)}|\\
&=& |D_{t,L}^{(3)}||D_{t,R}^{(3)}||C_t^{(3)}|.
\end{eqnarray*}
The equality of the lemma now follows from the above $5$ equalities for $M^{(i)}_t$.
 \end{proof}

\begin{cor} \label{rankatmost}
The algebra $\Br(\ddB_n)$ has a spanning set over $\Z[\delta^{\pm 1}]$ of
size at most $f(n)$.
 \end{cor}

\begin{proof}
By (\ref{eq:sum}),  the rank of $\Br(\ddB_n)$ is at most
$$|W(\ddB_n)|+\sum_{t=1}^{[\frac{n}{2}]}|D_{t,L}^{(1)}||D_{t,R}^{(1)}||C_t^{(1)}|
+\sum_{t=1}^{[\frac{n+1}{2}]}\left(\binom{n+1}{2t}t!!\right)^2 (n+1-2t)!$$
$$=2^{n}\cdot
n!+2^{n}\sum_{t=1}^{[\frac{n}{2}]}\left(\binom{n}{2t}t!!\right)^2 (n-2t)!
+\sum_{t=1}^{[\frac{n+1}{2}]}\left(\binom{n+1}{2t}t!!\right)^2 (n+1-2t)!.$$
From \cite{PWZ1996}, it follows that
$$\sum_{t=0}^{[\frac{k}{2}]}\left(\binom{k}{2t}t!!\right)^2 (k-2t)!=k!!.$$
By applying this for $k=n$, $n+1$ to the last two summands of the above
equality, we obtain that the rank of $\Br(\ddB_n)$ is at most $2^{n+1}\cdot
n!!- 2^{n}\cdot n!+(n+1)!!-(n+1)!=f(n)$.
\end{proof}

\np We end this section with a proof of Theorem \ref{thm:main}.  By
Corollary \ref{rankatmost} there is a spanning set of $\Br(\ddB_n)$ of size
$f(n)$.  By Theorem \ref{diagramimage}, this set maps onto a spanning set of
$\SBr(\ddD_{n+1})$ of size at most $f(n)$.  Moreover, by the same theorem,
$\SBr(\ddD_{n+1})$ is a free algebra over $\Z[\delta^{\pm 1}]$ of rank
$f(n)$. This implies that the spanning set of $\Br(\ddB_n)$ is a basis and
that $\Br(\ddB_n)$ is free of rank $f(n)$.  In particular,
$\phi:\Br(\ddB_n)\to \SBr(\ddD_{n+1})$ is an isomorphism and Theorem
\ref{thm:main} is proved.
\end{section}

\begin{section}{Cellularity}
Recall from \cite{GL1996} that an associative
algebra $\alg$ over a commutative ring $R$ is cellular if there is a quadruple
$(\Lambda, T, C, *)$ satisfying the following three conditions.

\begin{itemize}
\item[(C1)] $\Lambda$ is a finite partially ordered set.  Associated to each
$\lambda \in \Lambda$, there is a finite set $T(\lambda)$.  Also, $C$ is an
injective map
$$ \coprod_{\lambda\in \Lambda} T(\lambda)\times T(\lambda) \rightarrow \alg$$
whose image is called a \emph{cellular basis} of $\alg$.

\item[(C2)]
The map $*:\alg\rightarrow \alg$ is an
$R$-linear anti-involution such that
$C(x,y)^*=C(y,x)$ whenever $x,y\in
T(\lambda)$ for some $\lambda\in \Lambda$.

\item[(C3)] If $\lambda \in \Lambda$ and $x,y\in T(\lambda)$, then, for any
element $a\in \alg$,
$$aC(x,y) \equiv \sum_{u\in T(\lambda)} r_a(u,x)C(u,y) \
\ \ {\rm mod} \ \alg_{<\lambda},$$ where $r_a(u,x)\in R$ is independent of $y$
and where $\alg_{<\lambda}$ is the $R$-submodule of $\alg$ spanned by $\{
C(x',y')\mid x',y'\in T(\mu)\mbox{ for } \mu <\lambda\}$.
\end{itemize}
Such a quadruple $(\Lambda, T, C, *)$ is called a {\em cell datum} for
$\alg$.

\begin{thm}
\label{th:cellular}
There is a cellular datum for $\Br(\ddB_n)\otimes_{\Z[\delta^{\pm 1}]} R$ if
$R$ is an integral domain in which $2$ and $\delta$ are invertible elements.
\end{thm}

\begin{proof}
Let $R$ be as indicated and write
$\alg=\Br(\ddB_n)\bigotimes_{\Z[\delta^{\pm 1}]}R$. We introduce a
quadruple $(\Lambda, T, C, *)$ and prove that it is a cell datum for
$\alg$.  The map $*$ on $\alg$ will be the natural anti-involution
$\cdot^{\op}$ on $\alg$ over $R$.  By Proposition \ref{prop:opp}, the
natural anti-involution is an $R$-linear anti-involution of $\alg$.

By Theorem \ref{rewrittenforms} and Theorem
\ref{thm:main}, the Brauer algebra $\alg$ over $R$ has a basis consisting of
the elements of (i)--(vi) in Theorem \ref{rewrittenforms}.


For $t\in\{0,\ldots,\lfloor n/2\rfloor\}$,
let $C_t^*=\psi\phi(C^{(1)}_t)$ and
$C_t=\left<\psi(R_2),\psi\phi(C_t^{(2)})\right>\cong W(\ddA_{n-2t})\subset
W(\ddD_{n+1})$, and put
$Y=C_t^*$ or $C_t$.
As $Y$ is a Weyl group with irreducible factors of type $\ddB$ or
$\ddA$ and the coefficient ring $R$ satisfies the
conditions of \cite[Theorem 1.1]{G2007}, we conclude from \cite[Corollary
3.2]{G2007} that the group ring $R[Y]$ is a cellular
subalgebra of $\alg$.
Let $(\Lambda_Y, T_Y, C_Y, *_Y)$ be the corresponding
cell datum for $R[Y]$.  By \cite[Section 3]{G2007}, $*_Y$ is the map
$\cdot^{\op}$ on $R[Y]$. 

The underlying set $\Lambda$ is defined as the union of $\Lambda_1$
and $\Lambda_2$, where $\Lambda_1=\{t\}_{t=0}^{[\frac{n}{2}]}$,
$\Lambda_2=\{(t,\theta)\}_{t=1}^{[\frac{n+1}{2}]}$. A set of $t$ pairs
in $\{1,\ldots, n+1\}$ is called \emph{admissible} $t$-set in $\{1,\ldots,
n+1\}$ if no two pairs have a common number. We denote the set of all
admissible $t$-sets of $\{1,\ldots , n+1\}$ by $U_t^{n+1}$. A decorated
pair in $\{1,\ldots, n+1\}$ is a triple $\{i,j,+\}$ or $\{i,j,-\}$
with $1\leq i, j\leq n+1$ with $\pm$ for decorations.  A \emph{decorated
admissible} $t$-set in $\{1,\ldots, n+1\}$ is some admissible $t$-set
in $\{1,\ldots, n+1\}$ with each pair being decorated. We denote all
decorated admissible $t$-sets in $\{1,\ldots, n+1\}$ by
$U_t^{*n+1}$. The set of all decorated admissible $t$-sets in $\{1,\ldots,
n+1\}$ without $1$ appearing in any pair is denoted by $U_t^{|*n+1}$.
We view $U_t^{n+1}$ as the subset
of $U_t^{*n+1}$ of all admissible $t$-sets all of whose pairs are decorated
by $-$.  
For each $t\in \Lambda_1$, we define the associated
finite set to be
$$T(t)=\{(u,v)\mid u\in U_t^{|*n+1},\, v\in T_{C_t^*} \}.$$ For each
$(t,\theta)$, we define the associated finite set to be
$$T((t,\theta))=\{(u,v)\mid u\in U_t^{n+1},\, v\in T_{C_t} \}.$$ 
By Theorem \ref{diagramimage} and (C1) for $R(Y)$, there exists a map
$$D: \coprod_{\lambda\in \Lambda}T(\lambda)\times T(\lambda) \rightarrow
\BrMD(\ddB_n), $$ where the diagram of $D((u_1,v_1),(u_2,v_2))$ is given as
the top horizontal strands are strands between pairs in $u_1$ and decorated
for $+$, the bottom horizontal strands are similarly given by $u_2$ the free
(decorated) vertical strands and multiplied by $\xi \delta$ or not are given
by $C_{Y}(v_1, v_2)$, multiplied by $\theta\delta$ if $\lambda\in
\Lambda_2$.  We see $D$ is an injective map and its image is a basis of
$\BrMD(\ddB_n)$.  Therefore we define $C=\phi^{-1}\psi^{-1}D$. The partial
order on $\Lambda$ is given by
\begin{itemize}
\item $\lambda_1 >\lambda_2$ if
$\lambda_1=t_1<t_2=\lambda_2\in \Lambda_1$, \item (2) $\lambda_1 >\lambda_2$ if
$\lambda_1=(t_1,\theta)$, $\lambda_2=(t_2,\theta)\in \Lambda_2$ and
$t_1<t_2$, \item (3) $\lambda_1 >\lambda_2$ if $\lambda_1=t_1\in \Lambda_1$ and
$\lambda_2=(t_2,\theta)\in \Lambda_2$ and $t_1\leq t_2$.  
\end{itemize}
\noindent
It can be
illustrated by the following Hasse diagram, where $a>b$ is equivalent to the
existence of a directed path from $a$ to $b$.
\begin{center}
\phantom{longgggg}
\quad
\xymatrix{
0\ar[dr]\ar[r]&1\ar[r]\ar[d] & 2 \ar[d]\ar[r] &3\ar[d]\ar[r]&\cdots\ar[d]\\
         & (1,\theta) \ar[r] & (2,\theta)\ar[r]&(3,\theta)\ar[r]&\cdots
          }
\end{center}
In other words, we inherit the cellular structure of $\Br(\ddD_{n+1})$ in
\cite[section 6]{CFW2008}.  By Theorem \ref{thm:main} and Theorem
\ref{diagramimage}, the quadruple $(\Lambda, T, C, *)$ satisfies (C1).  From
the diagram representation of $\BrMD(\ddB_n)$ described in Theorem
\ref{diagramimage} and (C2) of $R[Y]$, the quadruple $(\Lambda, T, C, *)$
satisfies (C2) with $*=\cdot^{\op}$.  It remains to check condition (C3) for
$(\Lambda, T, C, *)$. For this we just need to consider $r_i
C((u_1,v_1),(u_2,v_2))$ and $e_i C((u_1,v_1),(u_2,v_2))$. This can be proved
by a case-by-case check using the lemmas in Section \ref{upperbound} or by
an argument using the diagram representation of $\BrMD(\ddB_n)$ and (C3) of
$R[Y]$.

We conclude that $(\Lambda, T, C, *)$ is a cell datum for
$\Br(\ddB_n)\otimes_{\Z[\delta^{\pm 1}]} R$.
\end{proof}

\np
\begin{rem} 
In \cite{KX1999}, K\"onig and Xi proved that Brauer algebras of type $\ddA$
are inflation cellular algebras, and also in \cite{BO2011}, Bowman proved
that the Brauer algebras found in \cite{CLY2010} cellularly stratified
algebras (a stronger version of inflation cellular algebras). Both kinds of
algebras have totally ordered sets $\Lambda$ associated to the cellular
structures. But Brauer algebras of type $\ddD$ and type $\ddB$, just have
partially ordered sets $\Lambda$ in the cell data given above. this explains
why we have not been able to use \cite{KX1999} for a cellularity proof.
\end{rem}
\end{section}

\end{document}